\documentclass[preprint]{imsart}

\RequirePackage{amsthm,amsmath}
\RequirePackage[numbers]{natbib}
\RequirePackage[colorlinks,citecolor=blue,urlcolor=blue]{hyperref}
\RequirePackage{graphicx}

\RequirePackage[OT1]{fontenc}
\usepackage{cancel}
\usepackage{graphicx,subcaption}
\usepackage{microtype}
\usepackage{yk}

\pubyear{2023}
\volume{0}
\issue{0}
\firstpage{1}
\lastpage{44}
\startlocaldefs
\numberwithin{equation}{section}
\theoremstyle{plain}

\makeatletter
\newcommand{\average}[2][0]{{%
		\mspace{#1mu}%
		\overline{\mspace{-#1mu}\average@check#2\relax}%
}}
\newcommand\average@check[1]{%
	#1\@ifnextchar_{\average@sub}{}%
}
\newcommand{\average@sub}[2]{% #1 is _
	_{#2}\mspace{-2mu}\aftergroup\average@compensate
}
\newcommand{\average@compensate}{\mspace{2mu}}
\makeatother

\def\ip#1{\langle #1 \rangle }

\def\hf{\widehat{f}}

\def\hlambda{\widehat{\lambda}}

\def\hT{\widehat{T}}

\def\hV{\widehat{V}}

\def\tf{\tilde{f}}
\def\tl{\tilde{l}}
\def\tg{\tilde{g}}
\def\tT{\widetilde{T}}
\def\tK{\widetilde{K}}
\def\tZ{\widetilde{Z}}

\def\tV{\widetilde{V}}

\def\pto{\overset{p}{\to}}
\def\reals{\mathbb{R}}
\newcommand{\floor}[1]{\lfloor #1 \rfloor}
\newtheorem*{ntheorem}{General Assumptions}
\newtheorem*{result1}{Result 1}
\newtheorem*{result2}{Result 2}
\endlocaldefs

\begin{document}

\begin{frontmatter}
% \title{A Sample Document\support{Support information of the article.}}
% \runtitle{A Sample Document}
\title{On uniform consistency of spectral embeddings}
\runtitle{Uniform consistency of spectral embeddings}

\begin{aug}
\author{\fnms{Ruofei} \snm{Zhao}\ead[label=e1]{rfzhao@umich.edu}},
\author{\fnms{Songkai} \snm{Xue}\ead[label=e2]{sxue@umich.edu}}
\and
\author{\fnms{Yuekai} \snm{Sun}\ead[label=e3]{yuekai@umich.edu}}

\address{University of Michigan \\
1085 S University,
Ann Arbor, MI 48109 \\
\printead{e1,e2,e3}}

\runauthor{R. Zhao et al.}

\end{aug}

\begin{abstract}
In this paper, we study the convergence of the spectral embeddings obtained from the leading eigenvectors of certain similarity matrices to their population counterparts. We opt to study this convergence in a uniform (instead of average) sense and highlight the benefits of this choice. Using the Newton-Kantorovich Theorem and other tools from functional analysis, we first establish a general perturbation result for orthonormal bases of invariant subspaces. We then apply this general result to normalized spectral clustering. By tapping into the rich literature of Sobolev spaces and exploiting some concentration results in Hilbert spaces, we are able to prove a finite sample error bound on the uniform consistency error of the spectral embeddings in normalized spectral clustering.
\end{abstract}

\begin{keyword}[class=MSC]
\kwd[Primary ]{62H30}
\kwd{47A55}
\end{keyword}

\begin{keyword}
	\kwd{Spectral embedding}
	\kwd{normalized spectral clustering}
	\kwd{uniform consistency}
	\kwd{the Newton-Kantorovich Theorem}
	\kwd{Sobolev spaces}
	\kwd{functional analysis}
	\kwd{concentration in Hilbert spaces}
\end{keyword}

\tableofcontents
\end{frontmatter}

\section{Introduction}

Spectral methods are a staple of modern statistics. For statistical learning tasks such as clustering or classification, one can featurize the data with spectral methods then perform the task on the features. In the past twenty years, spectral methods have seen wide applications in image segmentation \cite{shi2000imagesegmentation}, novelty detection \cite{hoffman2007novelty}, community detection \cite{Donetti2004commudetection}, bioinformatics \cite{higgs2006microarray}, and its effectiveness is partly credited to its ability to reveal the latent low-dimensional structure in the data. 

Spectral embedding gets its name from the fact that the embeddings are constructed from the spectral decomposition of a positive-definite matrix. For example, in normalized spectral clustering \cite{ng2001Spectral}, the {\it normalized Laplacian embedding} $\Phi_n:\{x_i\}_{i=1}^n\to\bbR^K$ is given by 
\begin{equation}
\Phi_n(x_i)^T = e_i^TV,\quad i\in[n],
\label{eq:normalizedLaplacianEmbedding}
\end{equation}
where $\{x_i\}_{i=1}^n$ are the observations, $e_i \in \bbR^n$'s are all zeros but one on the $i$-th entry, $K$ is the desired dimension of the embedding, and the columns of $V\in\bbR^{n\times K}$ are the leading eigenvectors of the normalized Laplacian matrix. As described, spectral embeddings are only defined on points in the training data, but it is possible to evaluate them on points that are not in the training data through out-of-sample extensions \cite{Bengio2003, vonluxburg2008Consistency}. Some other examples of spectral methods are Isomap \cite{tenenbaum2000Global}, Laplacian \cite{belkin2003Laplacian} and Hessian eigenmaps \cite{donoho2003Hessian}, and diffusion maps \cite{coifman2005Geometric}. 

Since downstream procedures take the embeddings as input, it is imperative that the embeddings have certain consistency properties to ensure the quality of the ultimate output. Specifically, we ask
\begin{itemize}
	\item In the large sample limit, do the embedded representations of the data ``converge'' to certain population level representations?
	\item If the embedded representations do converge, in what sense do they converge?
\end{itemize}

While there are many results on the convergence of eigenvalues and spectral projections, there are only a few results that directly address the convergence of the embedded representation in a general setting. The only exception is \citet{vonluxburg2008Consistency}. This is a gap in the literature because it is the embedded representation, not the spectral projections or the eigenvalues, that are the inputs to downstream application. In this paper, we address the two questions and provide direct answers --- we show the sample level embeddings converge uniformly to its population counterpart up to a unitary transformation. We improve the result of \citet{vonluxburg2008Consistency} by considering multidimensional embeddings and allowing for non-simple eigenvalues.

For a concrete application of our result, let us return to spectral clustering. The population counterpart of the normalized Laplacian embedding is given by
\[
\Psi(x)^T = \begin{bmatrix} f_1(x) & \dots & f_K(x) \end{bmatrix},
\]
where $f_1,\dots,f_K$ are the leading eigenfunctions of the normalized Laplacian operator \cite{vonluxburg2008Consistency,schiebinger2015geometry}. As is shown in \citet{vonluxburg2008Consistency}, the normalized Laplacian matrix has an operator counterpart  that we shall refer to as the empirical normalized Laplacian operator. Let $\hat{f}_{n,1},\dots,\hat{f}_{n,K}$ be the leading eigenfunctions of this operator, and define the embedding
\begin{equation}
\Psi_n(x)^T = \begin{bmatrix} \hat{f}_{n,1}(x) & \dots & \hat{f}_{n,K}(x) \end{bmatrix}.
\label{eq:def-psin}
\end{equation}
The embedding $\Psi_n$ coincides with $\Phi_n$ on the sample points, i.e.  $\Psi_n(x_i)=\Phi_n(x_i)$ for all $\{x_i\}_{i=1}^n$.  We shall show that the sample level embedding converges {\it uniformly} to its population counterpart:
\begin{equation}
\sup\{d(\Psi_n(x),\Psi(x)):x\in\cX\} \pto 0,
\label{eq:uniform_conv_def}
\end{equation}
where $d$ is some metric on $\bbR^K$. This implies $\Phi_n$ converges uniformly to the restriction of $\Psi$ to the sample points.

\subsection{Main results}
In this section, we state our results in an informal manner. These results are made precise and proved in subsequent sections.

Our first main result concerns the effect of perturbation on the invariant subspace of an operator. It serves as a general recipe for establishing uniform consistency type results. Although in statistics and machine learning, we mainly work with real-valued functions, our main spectral perturbation result is stated for complex-valued functions. This choice is technically convenient because the complex numbers are algebraically closed while the real numbers are not. In most applications of the result, the complex-valued functions only take real values.

Suppose $\cH$ is a complex Hilbert space whose elements are bounded complex-valued continuous functions over a domain $\cX$. Let $T,\tT$ be two operators from $\cH$ to $\cH$ that are close in Hilbert-Schmidt norm. Let $\{f_i\}_{i=1}^K$ be the top $K$ eigenfunctions of $T$ and $\{\tf_i\}_{i=1}^K$ be those of $\tT$. As long as $\{f_i\}_{i=1}^K$ and $\{\tf_i\}_{i=1}^K$ are appropriately normalized, we expect $\{f_i\}_{i=1}^K$ to be close to $\{\tf_i\}_{i=1}^K$ up to some unitary transformation. This is indeed the case and is characterized as follows by our first result.

\begin{result1}[General recipe for uniform consistency]
	Define $V_1: \bbC^K \to \cH$ as $ V_1 \alpha = \sum_{i=1}^{K} \alpha_i f_i$ and $\tV_1: \bbC^K \to \cH$ as $ \tV_1 \alpha = \sum_{i=1}^{K} \alpha_i \tf_i$. 
	There are constants $C_1,C_2 > 0$ that only depend on $T$ such that as long as $\|\tT-T\|_{HS} \leq C_1$, we have
	\begin{equation}
	\inf \{\|V_1 - \tV_1 Q\|_{2\to\infty}:Q\in\bbU^K\} \leq C_2 \|\tT-T\|_{HS},
	\label{eq:result1}
	\end{equation}
	where $\bbU^K$ is the space of unitary matrices in $\bbC^{K \times K}$, and the $2\to\infty$-norm of an operator $A:\bbC^K\to \cH$ is defined as
	\[
	\|A\|_{2\to\infty} = \sup\{\|A\alpha\|_\infty:\alpha\in\bbC^K, \|\alpha\|_2=1\}.
	\]
\end{result1}

It is not hard to notice the correspondence between \eqref{eq:result1} and \eqref{eq:uniform_conv_def}: $V_1$ is the analogue of $\Psi$; $\tV_1$ is the analogue of $\Psi_n$; the two to infinity norm guarantees the convergence is uniform, and the distance metric $d$ is chosen to measure Euclidean distance up to a unitary transformation.\footnote{Normalized eigenfunctions of the same eigenvalue are only determined up to unitary transformation.} This observation justifies naming $\inf \{\|V_1 - \tV_1 Q\|_{2\to\infty}:Q\in\bbU^K\}$ the uniform consistency error. It is also worth mentioning that the constant $C_2$ is inversely proportional to a measure of the eigengap between the $K$-th and $K+1$-th eigenvalues of $T$. This provides further justification for studying the convergence of the leading eigenspace as a whole, rather than studying the convergence of the individual eigenspaces like in \citet{vonluxburg2008Consistency}. Not only is the former more general and realistic, but it leads to better constants as well. In many applications, the top eigenvalues are usually clustered together, but there is a large gap between the top eigenvalues and the rest of the spectrum. Thus it is hard to estimate the corresponding eigenfunctions individually, but it is easy to estimate them altogether up to a unitary transformation.

Result 1 provides a general approach to proving uniform consistency: we simply need to bound the difference between the sample level operator and its population counterpart in an appropriate norm. The proof of Result 1 is also interesting in its own right. 
We identify the invariant subspace directly by solving an operator equation and appeal to the Newton-Kantorovich Theorem to characterize the solution. The main benefit of this approach is it overcomes the limitations of traditional approaches when working with non-unitarily invariant norms.

Our second main result is a finite sample uniform error bound for embedding in normalized spectral clustering. Let $C_b(\cX)$ denote the space of bounded continuous complex-valued functions over $\cX$. Define $V_1: \bbC^K \to C_b(\cX)$ as $ V_1 \alpha = \sum_{i=1}^{K} \alpha_i f_i$ where $f_1,\dots,f_K$ are the leading real-valued eigenfunctions of the normalized Laplacian operator. Define $\hV_{n,1}: \bbC^K \to C_b(\cX)$ as $ \hV_{n,1} \alpha = \sum_{i=1}^{K} \alpha_i \hf_{n,i}$ where $\hf_{n,i}$ are defined as in \eqref{eq:def-psin} and real-valued. Applying Result 1, we obtain
\begin{result2}[Uniform consistency for normalized spectral clustering]
	Under suitable conditions, there are constants $C_4,C_5 > 0$ that are independent of $n$ and the randomness of the sample, such that whenever the sample size $n \geq C_4 \tau$ for some $\tau >1$, we have 
	\[
	\inf \{\|V_1 - \hV_{n,1} Q\|_{2\to\infty}:Q\in\bbU^K\} \leq C_5 \frac{\sqrt{\tau}}{\sqrt{n}},
	\]
	with probability at least $1-8e^{-\tau}$.
\end{result2}

Despite the fact that that Result 2 is an application of Result 1, its proof is by no means simple. The main technical challenge is establishing concentration bounds for Hilbert-Schmidt operators. Result 2 suggests that the convergence rate, under appropriate conditions, is $\cO(\frac{1}{\sqrt{n}})$ (modulo a log factor). Moreover, in the context of clustering, the notion of uniform consistency leads to stronger assurances about the correctness of the clustering output. For example, in spectral clustering, the points are clustered based on their embeddings. Uniform convergence implies the embeddings of all points are close to their population counterpart. As long as the error in the embeddings are small enough, it is possible to show that all points are correctly clustered. This is not possible if the embeddings only ``converge in mean'': $\|V_1 - \hV_{n,1} Q\|_{2\to L^2(\cX, \bbP)} \to 0$.

\subsection{Related literature}
Most closely related to our results are the work of \citet{vonluxburg2008Consistency} and \citet{rosasco2010Learninga}. For normalized spectral clustering, \citet{vonluxburg2008Consistency} proved the convergence of the eigenvalues and spectral projectors of the sample level operator to their population counterparts. They also established uniform convergence of eigenfunctions whose corresponding eigenvalue has multiplicity one to their population counterparts. 
Our results are in the same vein as theirs in that we also study uniform convergence of eigenfunctions. We improve upon their uniform convergence result by considering multiple eigenfunctions at once and allowing for non-simple eigenvalues. 
In the context of unnormalized spectral clustering, \citet{rosasco2010Learninga} studied the convergence rate of the $l^2$-distance between the ordered spectrum of the sample level operator and that of the population operator and derived finite sample bound for the deviation between the sample level and population level spectral projections associated with the top $K$ eigenvalues. They also obtained finite sample spectral projection error bound for asymmetric normalized graph Laplacian. Our work is related to theirs because both study the convergence of the leading eigenspace and we owe much of our concentration results to them. The two works are also very distinct at the same time. Firstly, our notion of convergence is in uniform consistency of the eigenfunction, theirs is in terms of the induced RKHS norm between the spectral projectors. To the best of our knowledge, it is non-trivial to establish one set of results from the other. Secondly, we study the  normalized symmetric graph Laplacian, while they study the unnormalized graph Laplacian and asymmetric normalized graph Laplacian. 

The general relationship between the spectral properties of an empirical operator/matrix and that of its population counterpart has also been studied under other contexts. In \citet{koltchinskii2000}, it is proved that the ordered spectra of an integral operator and its empirical version tend to zero in $l^2$-distance almost surely if and only if the kernel is square integrable. Convergence rate and distributional limits were also obtained under stronger conditions. In \citet{koltchinskii1998}, the authors extended their own result by proving law of large numbers and central limit theorems for quadratic forms induced by spectral projections. The investigation of spectrum convergence is continued in \citet{mendelson2005} and \citet{mendelson2006}, where the authors associated various types of distance metric between two ordered spectra to the deviation of the sample mean of i.i.d rank one operators from its population mean. Similar problems have also been studied in kernel principal component analysis (KPCA) literature. For example, in \citet{shawetaylor2005} and \citet{blanchard2007}, the concentration property of the sum of the top $K$ eigenvalues and the sum of all but the top $K$ eigenvalues of the empirical kernel matrix are studied, because such partial sums are closely related to the reconstruction error of KPCA. In \citet{zwald2006}, a finite sample error bound on the difference between the projection operator to the leading eigenspace of the empirical covariance operator and that to the leading eigenspace of the population covariance operator is derived.  We remark that none of the results mentioned in this paragraph addressed the consistency of the embedding directly, nor did any consider kernel matrix normalized by the degree matrix.

Unlike our results which are model-agnostic, the property of spectral methods has also been studied in model-specific settings. For example, \citet{rohe2011} and \citet{lei2015} investigated the spectral convergence properties of the graph Laplacian and the consistency of spectral clustering in terms of community membership recovery under stochastic block models. When the data are sampled from a finite mixture of nonparametric distributions, \citet{shi2009} studied how the leading eigenfunctions and eigenvectors of the population level integral operator can reflect clustering information; \citet{schiebinger2015geometry} studied the geometry of the embedded samples generated by normalized spectral clustering and showed that the embedded samples for different clusters are approximately orthogonal when the mixtures have small overlapping and the sample size is large. We remark that in all the results mentioned in this section so far, the kernel function is fixed. For the relationship between the graph Laplacian and the Laplace-Beltrami operator on a manifold and the properties of spectral clustering when the kernel is chosen adaptively, we refer readers to the series of work by Trillos et. al. \cite{Trillos1,Trillos2,Trillos3} and the references therein.

Lastly, entrywise or row-wise analysis for eigevectors and eigenspaces of matrices has been studied in recent literature. For general purpose, deterministic $\ell_{2\to\infty}$ bounds are derived by \citet{fan2018eigenvector, cape2019two, damle2020uniform}, where the first two are for rectangular matrices and the last one is for symmetric matrices. When probabilistic assumptions are imposed on the true and perturbed matrices, \citet{cape2019signal, abbe2020entrywise, mao2021estimating} obtain stronger $\ell_{2\to\infty}$ bounds for various tasks by taking advantages of the structure of the random matrices. Comparing to these literature, we remark that our work provides a deterministic bound which aids in the $\ell_{2\to\infty}$ perturbation theory of linear operators, and the bound can be applied to many problems in statistics (\eg, spectral clustering and kernel PCA) for helping characterize the spectral embedding of individual samples.

\subsection{Main contributions}
We view our main contributions as three fold and list them in the order of appearance. First, we demonstrate that the Newton-Kantorovich Theorem provides a general approach to studying the effect of local perturbations on the invariant spaces of an operator. This result may be of independent interest to researchers working on spectral perturbation theory. Second, we study the convergence of the embeddings via uniform consistency error and offer a general recipe for establishing non-asymptotic uniform consistency type results that handles multiple eigenfunctions at once and is not limited to simple eigenvalues. Third, we apply our recipe to normalized spectral clustering and give a novel proof of finite sample error bound on the uniform consistency error of the spectral embeddings.

\subsection{Structure of the paper}
The rest of the paper is organized as follows: A review of relevant mathematical preliminaries is provided in Section 2; the exact statement and proof for Result 1 is in section 3; the exact statement and proof for Result 2 is in section 4; a discussion of various issues relevant to our results is in section 5; proofs of some secondary lemmas and an additional application are relegated to the appendix. 
\section{Preliminaries and notations}
In this section, we discuss various basic concepts and preliminary results that will be used repeatedly throughout the paper. More technical results that are section specific shall be introduced as needed later in the paper.

\subsection{Operator theory}

We assume readers are familiar with basic concepts such as Banach spaces, Hilbert spaces, linear operators, operator norms, and spectra of operators. From now on, we let $\bbK$ denote either the field of real numbers or the field of complex numbers, $\cY_1,\cY_2$ denote Banach spaces over the same field $\bbK$, and $\cH_1,\cH_2$ denote Hilbert spaces over the same field $\bbK$.

We would like to first highlight a nuance in the definition of linear operator. For a linear operator $A:\cY_1 \to \cY_2$, we adopt the convention from \citet{kato1980} and allow $A$ to be defined only on a linear manifold in $\cY_1$,\footnote{In \citet{kato1980}, linear manifold is just a synonym for affine subspace.} denoted $D(A)$.\footnote{In \citet{ciarlet2013}, for example, such a distinction isn't made.} We call $D(A)$ the {\it domain} of $A$ and can naturally define the {\it range} of $A$ as $R(A):=\{Ay \;|\; y \in D(A)\}$. As for $\cY_1,\cY_2$, we call them the {\it domain space} and the {\it range space} respectively. 

For a linear operator $A:\cY_1 \to \cY_2$, we say $A$ {\it bounded} if $\sup_{\|y\|_{\cY_1}=1} \|Ay\|_{\cY_2} < \infty$, and when $A$ is bounded, we define its {\it operator norm} $\|A\|:=\sup_{\|y\|_{\cY_1}=1} \|Ay\|_{\cY_2}$. Throughout the paper, when $\| \cdot \|$ has no subscript, it defaults to operator norm. We use $\cL(\cY_1,\cY_2)$ to denote the space of all bounded linear operators from $\cY_1$ to $\cY_2$. When $\cY_1 = \cY_2$, we simply write $\cL(\cY_1,\cY_1)$ as $\cL(\cY_1)$. We say $A$ is a {\it compact operator} if the closure of the image of any bounded set in $\cY_1$ under $A$ is compact. It is know that compact operators are bounded. 

For a bounded linear operator $A:\cH_1 \to \cH_2$, define its {\it adjoint} $A^*:\cH_2 \to \cH_1$ as the unique operator from $\cH_2$ to $\cH_1$ satisfying $\ip{Af,g}_{\cH_2}=\ip{f,A^*g}_{\cH_1}$ for $\forall f \in \cH_1,\forall g \in \cH_2$. Here, we use $\ip{\cdot,\cdot}_\cH$ to denote the inner product in the Hilbert space $\cH$. A basic property of $A^*$ is $\|A\|_{\cH_1 \to \cH_2} = \|A^*\|_{\cH_2 \to \cH_1}$, where the norm is operator norm and we use the $\cH_1 \to \cH_2$ notation to explicitly specify the domain space and range space. When $\cH_1 = \cH_2$, $A$ is called {\it self-adjoint} if $A$ is equal to its adjoint $A^*$, and $A$ is called {\it positive} if for any $f \in \cH_1$, $\ip{Af, f}_{\cH_1} \geq 0$.

We say a Hilbert space is {\it separable} if it has a basis of countably many elements. We say a bounded linear operator $A:\cH_1 \to \cH_2$ is {\it Hilbert-Schmidt} if $\sum_{i \in \cI} \|A e_i \|_{\cH_2}^2 < \infty$ where $\{e_i: i \in \cI\}$ is an orthonormal basis of $\cH_1$. We use $HS(\cH_1,\cH_2)$ to denote the space of all Hilbert-Schmidt operators from $\cH_1$ to $\cH_2$; this space is also a Hilbert space with respect to the inner product $\ip{A,B}_{HS}:=\sum_{i \in \cI} \ip{Ae_i,Be_i}_{\cH_2}$. We use $\|\cdot\|_{HS}$ to denote the norm induced by this inner product and note that all Hilbert-Schmidt operators are compact. We also note the Hilbert-Schmidt norm is stronger than operator norm in that $\|A\| \leq \| A\|_{HS}$, and the Hilbert-Schmidt norm is compatible with the operator norm in the following sense: for any Hilbert-Schmidt operator $A$ and bounded operator $B$, their product $AB$ and $BA$ are Hilbert-Schmidt and their Hilbert-Schmidt norm satisfies
\begin{gather*}
	\|AB\|_{HS} \leq \|A\|_{HS} \|B\|,\\
	\|BA\|_{HS} \leq \|B\| \|A\|_{HS}. 
\end{gather*}

\subsection{Spectral theory for linear operators}
\label{sec:spec-theory}
In this subsection, we set $\bbK=\bbC$. Let $A:\cH_1 \to \cH_1$ be a bounded linear operator. Similar to matrices, we say $\lambda \in \bbC$ is an {\it eigenvalue} of $A$ if for some {\it eigenvector} $f \in \cH_1$,
\[
Af = \lambda f \;\; \text{      and      } \;\; f \neq 0.
\]

In other words, $\lambda$ is an eigenvalue if the null space $N(\lambda I - A)$ is not $\{0\}$. We call $N(\lambda I - A)$ the {\it eigenspace} associated with $\lambda$, and the dimension of $N(\lambda I - A)$ is called the {\it geometric multiplicity} of $\lambda$. The spectrum of $A$ is defined as $\sigma(A):= \bbC \setminus \rho(A)$, where $\rho(A)$ is the {\it resolvent set}
\[
\rho(A) := \{ \lambda \in \bbC \;|\; (\lambda I - A)^{-1} \in \cL(\cH_1) \}.
\]

Eigenvalues are in the spectrum, but $\sigma(A)$ generally contains more than just eigenvalues. If $A$ is a compact operator, $\sigma(A)$ has the following structure: $\sigma(A) \setminus \{0\}$ is a countable set of isolated eigenvalues, each with finite geometric multiplicity, and the only possible accumulation point of $\sigma(A)$ is $0$. If $A$ is self-adjoint, then all the eigenvalues must be real. If $A$ is a positive operator, then all its eigenvalues are real and non-negative. Therefore for any compact positive self-adjoint operator, we can arrange the non-zero eigenvalues of $A$ into a non-increasing sequence of positive numbers,\footnote{Because the largest eigenvalue is bounded by the operator norm of $A$.} and repeat each eigenvalue for a number of times equal to its geometric multiplicity.

Another remarkable fact in the spectral theory of linear operators concerns spectral projection. Let $\Gamma \subset \rho(A)$ be a closed simple rectifiable curve. Assume the part of $\sigma(A)$ enclosed inside $\Gamma$ is a finite number of eigenvalues $\lambda_1,\lambda_2,\ldots,\lambda_K$. Then the projection $P:\cH_1 \to \cH_1$ which projects to the direct sum of the eigenspaces of $\{\lambda_i\}_{i=1}^K$, i.e. $\bigoplus_{i=1}^K N(\lambda_i I - A)$, can be defined. Technicalities aside, this projection has the following contour integration expression
\[
P = \frac{1}{2\pi i} \int_{\Gamma} (\gamma I - A)^{-1} d \gamma.
\]

\subsection{Function spaces}
Let $\cX$ be a bounded open subset of $\bbR^p$, we now define several function spaces we are going to work with. Define the space of bounded continuous functions $C_b(\cX)$ as
\[
C_b(\cX) := \{ f \;|\; f:\cX \to \bbC \text{ is a bounded continous function}\}.
\]
It can be shown that $\|f\|_\infty:= \sup_{x \in \cX} |f(x)|$ is a norm on $C_b(\cX)$ and $C_b(\cX)$ is a Banach space with respect to this infinity norm.

We can also define the space of complex-valued square integrable functions $L^2(\cX, \mu)$. Suppose $(\cX, \cB, \mu)$ is a measure space where $\cB$ is the Lebesgue $\sigma$-algebra and $\mu$ is a measure, then $L^2(\cX, \mu)$ is defined as the set of measurable functions such that
\[
\int_\cX |f(x)|^2 d\mu < \infty.
\]
In fact, $L^2(\cX, \mu)$ is a Hilbert space with respect to the inner product
\[
\ip{f,g}_{L^2(\cX, \mu)}:= \int_\cX f \bar{g} d\mu.
\]

We also define $l^2$, the space of square summable infinite sequence of complex numbers. It is well known that $l^2$ is a complex Hilbert space with respect to the inner product
\[
\ip{u,v}_{l^2}:= \sum_{i=1}^{\infty} u_i \bar{v_i}.
\]

\subsection{Reproducing Kernel Hilbert Space (RKHS)}
Let $\cX$ be a subset of $\bbR^p$ and $\cH$ be a set of functions $f:\cX \to \bbC$. Suppose $\cH$ is a Hilbert space with respect to some inner product $\ip{\cdot, \cdot}_{\cH}$. If in $\cH$, all point evaluation functionals are bounded, i.e.
\[
|f(x)| \leq C_x \|f\|_{\cH} \quad \forall f \in \cH,
\]
where $C_x$ is some constant depending on $x$, then it can be shown that there exists a unique conjugate symmetric positive definite kernel function $k:\cX \times \cX \to \bbC$, such that the following reproducing property is satisfied:
\[
f(x) = \ip{f, k(\cdot, x)}.
\]
The kernel $k$ is called the reproducing kernel and $\cH$ is called a {\it reproducing kernel Hilbert space (RKHS)}.

We say a kernel function $k$ is {\it positive definite} if for any $n \in \bbN^+$, any $x_1,x_2,\ldots,x_n \in \cX$ and any $\xi_1,\xi_2,\ldots,\xi_n \in \bbC$, the quadratic form 
\begin{equation}
\sum_{i,j=1}^n k(x_i,x_j)\bar{\xi_i}\xi_j
\end{equation}
is non-negative. The kernel function for any RKHS is positive definite.

\section{Uniform error bound for spectral embedding}
In this section, we prove the first result described in the previous section. We first lay out the assumptions and notations. Let $\cX$ be a subset of $\bbR^p$ and $\bbP$ be a probability measure whose density function is supported on $\cX$. Let $L^2(\cX, \bbP)$ denote the space of complex-valued square integrable functions on $\cX$ and $\cH$ be a subspace of $L^2(\cX, \bbP)$. We assume $\cH$ is equipped with its own inner product $\ip{\cdot,\cdot}_\cH$ and is a Hilbert space with respect to this inner product. We also require $\cH$ to be such that for every $h \in \cH$, which is an equivalent class in $L^2(\cX, \bbP)$, there exists a representative function $h'$ in the class such that $h' \in C_b(\cX)$. Since $\supp (\bbP)= \cX$, $h'$ is unique, and we can define infinity norm on $\cH$ by setting $\|h\|_\infty := \|h'\|_\infty$. We require that on $\cH$, the norm induced by the $\cH$-inner product, denoted $\| \cdot \|_\cH$, be stronger than the infinity norm; that is there is a constant $C_{\cH} > 0$ such that $\|f\|_\infty \leq C_\cH \|f\|_\cH$ for all $f \in \cH$.\footnote{This in fact implies $\cH$ is an RKHS. But since we do not use the reproducing property anywhere in the proof, we find framing $\cH$ as an RKHS unnecessary.} 

Let $T$ and $\tT$ be two Hilbert-Schmidt operators from $\cH$ to $\cH$; $\tT$ can be seen as a perturbed version of $T$ and we use $E:=\tT-T$ to denote their difference. Suppose all the eigenvalues of $T$ (counting geometric multiplicity) can be arranged in a non-increasing (possibly infinite) sequence of non-negative real numbers $\lambda_1 \geq \lambda_2 \geq \ldots \geq \lambda_K > \lambda_{K+1} \geq \ldots \geq 0$ with a positive gap between $\lambda_K$ and $\lambda_{K+1}$. Suppose the eigenvalues of $\tT$ can also be arranged in a non-increasing sequence of non-negative real numbers. We do not assume, however, any eigengap for $\tT$.

Let $\{f_i\}_{i=1}^K \subset \cH$ be the eigenfunctions associated with eigenvalues $\lambda_1,\ldots,\lambda_K$. We assume $\{f_i\}_{i=1}^K$ are so picked that they constitute a set of orthonormal vectors in $L^2(\cX, \bbP)$. We then pick $\{f_i\}_{i=K+1}^\infty$ so that $\{f_i\}_{i=1}^\infty$ constitute a complete orthonormal basis of $L^2(\cX, \bbP)$. Define $V_1:\bbC^K \to L^2(\cX, \bbP)$ by $V_1 \alpha = \sum_{i=1}^{K} \alpha_i f_i$ and $V_2: l^2 \to  L^2(\cX, \bbP)$ by $V_2 \beta = \sum_{i=1}^{\infty} \beta_i f_{K+i}$. Define their adjoints $V_1^*, V_2^*$ with respect to the standard inner product on $\bbC^K, l^2$ and $L^2(\cX, \bbP)$. Since $\{f_i\}_{i=1}^K \subset \cH$, we can also view $\cH$ as the range (domain) space of $V_1$ ($V_1^*$). The exact range space of $V_1$ shall be clear from the context.

When the perturbation $E$ has small enough Hilbert-Schmidt norm, $\tT$ necessarily has an eigengap. In this case, the leading $K$-dimensional invariant subspace of $\tT$ is well defined. We pick $\{\tf_i\}_{i=1}^K$ to be an orthonormal set of vectors in $L^2(\cX, \bbP)$ such that they span the leading invariant subspace of $\tT$, and define $\tV_1:\bbC^K \to L^2(\cX, \bbP)$ as $\tV_1 \alpha = \sum_{i=1}^{^K} \alpha_i \tf_i$.

Last but not least, define $V_2^{-1}\cH := \{ l \in l^2 \,	:\, V_2 l \in \cH \}$ and $V_2^*\cH:=\{ V_2^*h  \,	:\, h \in \cH \}$. They are intuitively the ``coordinate space'' for functions in $\cH$ under the basis in $V_2$. Working with these coordinates could simplify our notations. The following facts regarding $V_2^{-1}\cH$ and $V_2^*\cH$ hold true (with proof in appendix).
\begin{lemma}
	Assuming $f_1,\dots,f_K \in \cH$, we have
	\begin{enumerate}
		\item the set $V_2^{-1}\cH$ is equal to the set $V_2^* \cH$;
		
		\item $V_2^{-1}\cH$ is a subspace of $l^2$; it is also a Hilbert space with respect to the $\cH$-induced inner product 
		\[(b_1, b_2)_{V_2^{-1}\cH}:= \ip{V_2 b_1, V_2 b_2}_\cH;\]
		
		\item $V_1 \in \cL(\bbR^K, \cH)$, $V_1^* \in \cL(\cH, \bbR^K)$, $V_2 \in \cL(V_2^{-1}\cH, \cH)$, and $V_2^* \in \cL(\cH, V_2^{-1}\cH)$, with operator norms satisfying
		\begin{gather*}
		\|V_1\|_{2 \to \cH} \leq \sqrt{K}\max_{i \in [K]} \|f_i\|_\cH, \quad \|V_1^*\|_{\cH \to 2} \leq C_\cH\sqrt{K}, \\
		\|V_2\|_{V_2^{-1}\cH \to \cH} =1, \quad \|V_2^*\|_{\cH \to V_2^{-1}\cH} \leq 1 + C_\cH K\max_{i \in [K]} \|f_i\|_\cH.
		\end{gather*}
	\end{enumerate}
	\label{lm:induced_hilbert}
\end{lemma}

Because of item 1 of the lemma, we do not need to distinguish between $V_2^{-1}\cH$ and $V_2^* \cH$; we denote both by $\tl^2$. To keep notation manageable, define $\tT_{ij} = V_i^*\tT V_j$ for any $i,j \in \{1, 2\}$; \eg\ $\tT_{21}$ is shorthand for $V_2^*\tT V_1$. 

We also need the following quantities to define the constants in Result 1. Let $\Gamma$ be the boundary of the rectangle
\begin{equation}
\big\{ \lambda \in \bbC \;|\; \frac{\lambda_K + \lambda_{K+1}}{2} \leq re(\lambda) \leq \|T\|_{\cH \to \cH} +1, |im(\lambda)|\leq 1 \big\}.
\label{eq:def-gamma}
\end{equation}
Let $l(\Gamma)$ denote the length of $\Gamma$ and define
\begin{equation}
\eta = \frac{1}{\sup_{\lambda \in \Gamma} \| (\lambda I - A)^{-1} \|_{op}},
\label{eq:def-eta}
\end{equation}
which is necessarily finite. Define a measure of spectral separation
\[
\delta:=\text{sep}(T_{11}, T_{22}) := \inf \Big\{ \big\| T_{22}Y-YT_{11} \big\|_{HS} \; \Big| \; Y \in \cL(\bbC^K, \tl^2), \|Y\|_{HS}=1 \Big\}.
\]
It is reasonable to expect that the larger the eigengap, the larger the $\text{sep}(T_{11}, T_{22})$. And when $T$ has only $K$ eigenvalues or is self-adjoint from $\cH$ to $\cH$, it is provably so that the separation $\text{sep}(T_{11}, T_{22})$ is lower bounded by the eigengap $\lambda_K -\lambda_{K+1}$.

Define constant
\[
C_3 := \max \Big \{C_\cH, 1 + C_\cH \|V_1\|_{2 \to \cH}, \|V_1\|_{2 \to \cH}(1 + C_\cH \|V_1\|_{2 \to \cH})  \Big \}.
\]
We are now ready to state the main theorem of this section.

\begin{theorem}[General recipe for uniform consistency]
	Under the assumptions above, as long as $E:=\tT-T$ as an operator from $\cH$ to $\cH$ has Hilbert Schmidt norm $\|E\|_{HS} \leq C_1$, the uniform consistency error satisfies
	\[
	\inf \{\|V_1 - \tV_1 Q\|_{2\to\infty}:Q\in\bbU^K\} \leq C_2 \|E\|_{HS}.
	\]
	Here, $C_1,C_2 >0$ are two constants independent of the choice of $\tT$ defined as
	\begin{gather}
	    C_1 := \frac{1}{C_3} \min \Big\{ \frac{\lambda_{K}-\lambda_{K+1}}{8}, \frac{1}{2}, \frac{\delta}{4}, \frac{\delta}{4 C_\cH}, \frac{C_3\eta ^2}{\eta + l(\Gamma)/2\pi} \Big\}, \label{eq:def-C1} \\
	    C_2 := \frac{4C_3C_\cH (\|V_1\|_{2\to\infty} +1)}{\delta}.\label{eq:def-C2}
	\end{gather}
	\label{thm:main-thm}
\end{theorem}

We remark that since $C_2$ is inversely proportional to $\delta$, it is beneficial to study the convergence of the leading eigenspace as a whole. When eigenspaces are treated individually, each eigenspace converges slowly because the leading eigenvalues may cluster together and we have a small $\delta$, but when treated as a whole, we get a larger $\delta$ and thus faster convergence because the leading eigenvalues are well-separated from the rest of the spectrum.

The rest of the section is devoted to proving Theorem \ref{thm:main-thm}. The proof strategy is to express $\tV_1 Q$ in terms of the solution of an operator equation and directly bound $\|V_1 - \tV_1 Q\|_{2\to\infty}$. We present the proof in five steps. In step one, we characterize the invariant subspace of $\tT$ in terms of the solution of a quadratic operator equation. In step two, we apply the Newton-Kantorovich Theorem to show this equation does have a solution when the perturbation is small. In step three, we introduce some additional conditions that guarantee the invariant space from step two is the leading invariant space. In step four, we directly bound the error term $\|V_1 - \tV_1 Q\|_{2\to\infty}$. In step five, we assemble all pieces together and prove Theorem \ref{thm:main-thm}. A similar approach was used in \citet{stewart71} to study the invariant subspace of matrices.

\subsection{Step one: equation characterization of the invariant subspace}
In this section, our goal is to find a $Y \in \cL(\bbC^K, \tl^2)$ such that the range of $V_1+V_2 Y \in \cL(\bbC^K, \cH)$ is an invariant subspace of $\tT$. It turns out that any $Y$ that satisfies the following quadratic operator equation suffices.

\begin{proposition}
	As long as $Y \in \cL(\bbC^K, \tl^2)$ satisfies the equation
	\begin{equation}
	\tT_{21} + \tT_{22} Y = Y\tT_{11} + Y \tT_{12} Y,
	\label{eq:quad_eq}
	\end{equation}
	the range of $V_1+V_2 Y \in \cL(\bbC^K, \cH)$ is an invariant subspace of $\tT$. 
\end{proposition}
\begin{proof}
	First, we note \eqref{eq:quad_eq} is a well-defined equation of operators in $\cL(\bbC^K, \tl^2)$. This can be seen from our assumption $\tT \in \cL(\cH)$ and item 3 of Lemma \ref{lm:induced_hilbert}. Next, we assert that equation \eqref{eq:quad_eq} implies\footnote{We do not differentiate equality in $L^2(\cX, \bbP)$ from equality in $\cH$, because the two are equivalent.}
	\begin{equation}
	\tT (V_1 + V_2 Y)  =(V_1 + V_2 Y) (\tT_{11} + \tT_{12} Y),
	\end{equation}
	which suggests that the range of $V_1 + V_2 Y$ is invariant under $\tT$. To prove this assertion, note
	\begin{align*}
	\tT (V_1 + V_2 Y) &{ = } (V_1V_1^* + V_2 V_2^*)\tT (V_1 + V_2 Y) \\
	&= V_1\tT_{11} + V_1\tT_{12}Y + V_2 \tT_{21} + V_2 \tT_{22} Y \\
	&= V_1\tT_{11} + V_1\tT_{12}Y + V_2 Y\tT_{11} + V_2 Y \tT_{12} Y \\
	&= (V_1 + V_2 Y) (\tT_{11} + \tT_{12} Y).
	\end{align*}
\end{proof}

\subsection{Step two: solve the equation with the Newton-Kantorovich Theorem}
After characterizing the invariant subspace of $\tT$ in terms of a solution of  \eqref{eq:quad_eq}, we apply the Newton-Kantorovich Theorem to prove a solution to \eqref{eq:quad_eq} exists. The Newton-Kantorovich Theorem constructs a root of a function between Banach spaces when certain conditions on the function itself and its first and second order derivatives are met. The construction is algorithmic: the root is the limit point of a sequence of iterates generated by the Newton-Raphson method for root finding. The exact version of the Newton-Kantorovich Theorem we use is from the appendix of \citet{karow2014}.

\begin{theorem}[Newton-Kantorovich]
	Let $\cE,\cZ$ be Banach spaces and let $F:\cZ \to \cE$ be twice continuously differentiable in a sufficiently large neighborhood $\Omega$ of $Z \in \cZ$. Suppose that there exists a linear operator $\bbT:\cZ \to \cE$ with a continuous inverse $\bbT^{-1}$ and satisfying the following conditions:
	\begin{gather}
	\| \bbT^{-1}(F(Z)) \|_\cZ \leq a, \\
	\| \bbT^{-1}\circ F'(Z) - I \|_{op} \leq b, \\
	\| \bbT^{-1}\circ F''(\tZ) \|_{op} \leq c, \qquad \forall \tZ \in \Omega.
	\end{gather}
	If $b < 1$ and $h:=\frac{ac}{(1-b)^2} < \frac{1}{2}$, then there exists a solution $Z_E$ of $F(Z_E)=0$ such that 
	\[
	\|Z_E - Z\|_\cZ \leq r_0 \qquad \text{with} \qquad r_0:= \frac{2 a}{(1-b)(1 + \sqrt{1-2h})}.
	\]
\end{theorem}

We are now ready to prove the proposition below, which states that when $\|E_{11}\|_{HS}$, $\|E_{21}\|_{HS}$, $\|E_{22}\|_{HS}$, $\|E_{12}\|_{HS}$ are small relative to $\text{sep}(T_{11}, T_{22})$, equation \eqref{eq:quad_eq} has a solution.

\begin{proposition}
	Let $\delta := \text{sep}(T_{11}, T_{22})$ and $s_E := \delta - \| E_{22} \|_{HS} -  \| E_{11} \|_{HS}$. When $s_E > 0$ and $\frac{\| E_{21} \|_{HS}   \| E_{12} \|_{HS}}{s_E^2} < \frac{1}{4}$, there exists $Y \in \cL(\bbC^K, \tl^2)$ with $\|Y\|_{HS} \leq \frac{2\| E_{21} \|_{HS}}{s_E}$ such that equation \eqref{eq:quad_eq} is satisfied.
	\label{lm:nk}
\end{proposition}

\begin{proof}[Proof of Proposition \ref{lm:nk}]
	After rearrangement, \eqref{eq:quad_eq} is equivalent to
	\begin{equation}
	E_{21} + (T_{22}+E_{22})Y - Y(T_{11}+E_{11}) - YE_{12}Y = 0.
	\label{eq:condition-on-y}
	\end{equation}
	Let $\cE:= \cL(\bbC^K, \tl^2)$ denote the space of bounded linear operators from $\bbC^K$ to $\tl^2$. Since $\bbC^K$ is finite dimensional, any linear operator from $\bbC^K$ to $\tl^2$ is bounded and Hilbert Schmidt. We can thus use Hilbert Schmidt norm as the default norm on $\cE$ and $\cE$ is a Hilbert space with respect to this norm. This fact also allows us to define $f:\cE \mapsto \cE$ as $f(Y) := E_{21} + (T_{22}+E_{22})Y - Y(T_{11}+E_{11}) - YE_{12}Y$ and $\cT:\cE \mapsto \cE$ as $\cT(Y):=T_{22}Y - YT_{11}$. Noting that the image of $\cH$ under $T,E$ are still in $\cH$, we can verify $\cT$ and $f$ are indeed well defined.
	
	We assert $\delta > 0$ and $\cT$ is one-to-one and onto and defer the proof of this to a lemma. The implication of this is $\cT$ is invertible with $\|\cT^{-1}\|_{op}=\frac{1}{\text{sep}(T_{11}, T_{22})} \leq \frac{1}{\delta}$.
	
	We are now ready to verify the three assumptions of Newton-Kantorovich theorem.

	\textbf{(A1):}
	\begin{equation}
	\big\| \cT^{-1}(f(0)) \big\|_{HS} = \big\| \cT^{-1}(E_{21}) \big\|_{HS} \leq \big\| \cT^{-1} \big\|_{op} \big\|E_{21} \big\|_{HS} \leq \frac{\big\|E_{21} \big\|_{HS}}{\delta} =:a.
	\end{equation}
	
	\textbf{(A2):}
	The Fr\'echet derivative of $f$ at $Y_0$ is given by
	\begin{align*}  
	f'(Y_0)\colon \cE & \longrightarrow  \cE \\
	\Delta Y & \longmapsto (T_{22}+E_{22})\Delta Y - \Delta Y(T_{11}+E_{11})\\
	&\quad \quad - \Delta Y E_{12}Y_0 - Y_0 E_{12}\Delta Y.
	\end{align*}
	
	Especially, when $Y_0=0$,
	\begin{align*}  
	f'(0)\colon \cE & \longrightarrow \cE \\
	\Delta Y & \longmapsto (T_{22}+E_{22})\Delta Y - \Delta Y(T_{11}+E_{11}).
	\end{align*}
	
	Consequently,
	\begin{align*}  
	\big( \cT^{-1} \circ f'(0)-I \big) \colon \cE & \longrightarrow \cE \\
	\Delta Y &\longmapsto \cT^{-1}(E_{22} \Delta Y - \Delta Y E_{11}).
	\end{align*}
	
	We thus have
	\begin{align*}
	\big\| \cT^{-1} \circ f'(0)-I  \big\| &= \sup_{\|\Delta Y \|_{HS}=1} \big\|\cT^{-1}(E_{22} \Delta Y - \Delta Y E_{11}) \big\|_{HS} \\
	&\leq \sup_{\|\Delta Y \|_{HS}=1} \big\| \cT^{-1} \big\|_{op} \big\|E_{22} \Delta Y - \Delta Y E_{11} \big\|_{HS} \\
	&\leq \frac{1}{\delta} \sup_{\|\Delta Y \|_{HS}=1} \bigg\{  \big\| E_{22} \big\|_{HS} \big\| \Delta Y \big\|_{HS} + \big\| \Delta Y \big\|_{HS} \big\| E_{11} \big\|_{HS} \bigg\} \\
	&\leq \frac{ \| E_{22} \|_{HS} +  \| E_{11} \|_{HS}}{\delta}:=b.
	\end{align*}
	
	\textbf{(A3):} The second order Fr\'echet derivative at $Y_0$ is a linear operator in $\cL(\cE, \cL(\cE,\cE))$:
	\begin{align*}  
	f''(Y_0) \colon \cE & \longrightarrow \cL(\cE,\cE)\\
	\Delta_1 Y &\longmapsto \cT_{\Delta_1 Y},
	\end{align*}
	where $\cT_{\Delta_1 Y}$ is
	\begin{align*}  
	\cT_{\Delta_1 Y} \colon \cE & \longrightarrow \cE \\
	\Delta_2 Y &\longmapsto -\Delta_1 Y E_{12} {\Delta_2 Y} - \Delta_2 Y E_{12} {\Delta_1 Y}.
	\end{align*}
	
	Therefore the second derivative is a constant for every $Y_0 \in \cE $ and we have,
	\begin{align*}  
	\big\| \cT^{-1} \circ f''(Y_0)  \big\|_{op} &= \sup_{\|\Delta_1 Y \|_{HS}=1} \big\| \cT^{-1} \circ \cT_{\Delta_1 Y}  \big\|_{op} \\
	&= \sup_{\|\Delta_1 Y \|_{HS}=1} \sup_{\|\Delta_2 Y \|_{HS}=1} \big\| \big( \cT^{-1} \circ \cT_{\Delta_1 Y} \big)(\Delta_2 Y)  \big\|_{\cE} \\
	&\leq \frac{1}{\delta} \sup_{\|\Delta_1 Y \|_{HS}=1} \sup_{\|\Delta_2 Y \|_{HS}=1} \big\| \Delta_1 Y E_{12} {\Delta_2 Y} + \Delta_2 Y E_{12} {\Delta_1 Y}  \big\|_{\cE}\\
	&\leq \frac{2\|E_{12}\|_{HS}}{\delta}:=c
	\end{align*}
	
	\textbf{(Conclusion:)} With all assumptions in place, we apply the Newton-Kantorovich Theorem and conclude as follows. When
	\begin{align*}
	&s_E := \delta - \| E_{22} \|_{HS} -  \| E_{11} \|_{HS} > 0 \text{ and }\\
	&\frac{\| E_{21} \|_{HS}   \| E_{12} \|_{HS}}{s_E^2} < \frac{1}{4},
	\end{align*}
	equation \eqref{eq:quad_eq} has solution $Y_E$ such that
	\begin{equation}
	\big\| Y_E \big\| \leq \frac{2\| E_{21} \|_{HS}}{s_E + \sqrt{s_E^2 - 4 \| E_{21} \|_{HS}   \| E_{12} \|_{HS} }} \leq \frac{2\| E_{21} \|_{HS}}{s_E}.
	\label{eq:yebound}
	\end{equation}
\end{proof}

\subsection{Step three: showing the invariant space is the leading eigenspace}
In step two, we obtained an invariant subspace, but there is no guarantee that the invariant subspace we obtained is the leading $K$-dimensional invariant subspace of $\tT$. In this subsection, we give sufficient conditions to ensure this. When $\|E\|_{HS}$ is small, we show several things must happen: first, the range of $V_1+V_2Y_E$ is $K$ dimensional; second, the eigenvalues of the restriction of $\tT$ to this subspace are contained in the interval $[\lambda_K -\epsilon, \lambda_1+\epsilon]$ for some small $\epsilon$; third, $\tT$ has exactly $K$ eigenvalues (counting geometric multiplicity) in the interval $[\lambda_K -\epsilon, \lambda_{max}(\tT)+\epsilon]$. These facts combined implies the invariant subspace from Proposition \ref{lm:nk} has to be the leading $K$-dimensional invariant subspace.

The first point is not hard to show. Suppose the range of $V_1+V_2 Y_E$ has less than $K$ dimensions, then there exists $s \in \bbC^K$ with $\|s\|_2=1$ such that $(V_1 +V_2Y_E)s=0$. But since $\|(V_1 +V_2Y_E)s\|_{L^2} \geq \|V_1s\|_{L^2}- \|V_2Y_Es\|_{L^2} \geq 1 - C_\cH\|V_2Y_Es\|_{\cH} \geq 1 - C_\cH \|Y_E\|_{HS}$, when $\|Y_E\|_{HS}$ is small, the vector $(V_1 +V_2Y_E)s$ simply cannot be zero. Stated formally, we have
\begin{lemma}
	When $C_\cH \|Y_E\|_{HS} < 1$, the range of $V_1+V_2 Y_E$ is $K$-dimensional.
	\label{lm:kdimensional}
\end{lemma}

As for the second point, which is to determine the eigenvalues of the restriction of $\tT$, note that \eqref{eq:quad_eq} implies $\tT$ has matrix representation $\tT_{11} + \tT_{12}Y_E$ in the basis $V_1+V_2 Y_E$. Since eigenvalues are not affected by the choice of bases, we know the eigenvalues of $T$ on the invariant space are those of $\tT_{11} + \tT_{12}Y_E$. Next, we recall a perturbation result for eigenvalues.

\begin{lemma}
	Assuming $\sigma(T_{11} + E_{11} + E_{12}Y_E) \subset \bbR$, we have (addition is set addition)
	\begin{align}
	\sigma(T_{11} + E_{11} + E_{12}Y_E) \subset \sigma(T_{11}) + \big[-\|E_{11} + E_{12}Y_E\|, \|E_{11} + E_{12}Y_E\| \big]. \label{eq:perturb1}
	\end{align}
	\label{lm:eigenoninvar}
\end{lemma}

\begin{proof}
	First note that $T_{11} = \text{diag}\big( (\lambda_1,\lambda_2,\ldots, \lambda_K)^T \big)$. Suppose $\lambda$ is a real eigenvalue of $T_{11} + E_{11} + E_{12}Y_E \in \bbC^{K \times K}$, then there exists $v \in \bbC^K$ with $\|v\|_2=1$ such that $(T_{11} + E_{11} + E_{12}Y_E)v = \lambda v$. It thus follows
	\[
	\inf_{i \in [K]} |\lambda_i - \lambda| \leq \Big(\sum_{i=1}^K (\lambda_i - \lambda)^2 |v_i|^2 \Big)^{1/2} = \|(\lambda I - T_{11}) v\| = \| (E_{11} + E_{12}Y_E)v \| \leq \| E_{11} + E_{12}Y_E\|,
	\]
	which suggests $\lambda$ is within $\| E_{11} + E_{12}Y_E\|$ from at least one of $\{\lambda_i\}_{i=1}^K$. This is equivalent to the claim of \eqref{eq:perturb1}.
\end{proof}

From the lemma, we know (assuming $\|Y_E\|_{HS} < 1 $)
\begin{equation}
\sigma(T_{11} + E_{11} + E_{12}Y_E) \subset \big[\lambda_K - \|E_{11}\|_{HS} - \|E_{12}\|_{HS},  \lambda_1 + \|E_{11}\|_{HS} + \|E_{12}\|_{HS} \big].
\label{eq:NKsubspace-eigenbound}
\end{equation}

For the third point, we need the following result from \citet{rosasco2010Learninga} (their Theorem 20), which they credited to \citet{anselone1971} for the origin.

\begin{theorem}
	Let $A\in\cL(\cH)$ be a compact operator. Given a finite set $\Lambda$ of non-zero eigenvalues of $A$, let $\Gamma$ be any simple rectifiable closed curve (having positive direction) with $\Lambda$ inside and $\sigma(A) \setminus \Lambda$ outside. Let $P$ be the spectral projection associated with $\Lambda$, that is,
	\begin{equation}
	P = \frac{1}{2\pi i} \int_{\Gamma} (\lambda I - A)^{-1} d \lambda,
	\end{equation}
	and define
	\begin{equation}
	\eta^{-1} = \sup_{\lambda \in \Gamma} \| (\lambda I - A)^{-1} \|_{op}.
	\label{eq:def-eta}
	\end{equation}
	Let $B$ be another compact operator such that
	\begin{equation}
	\| B - A \|_{op} \leq \frac{\eta ^2}{\eta + l(\Gamma)/2\pi} < \eta,
	\end{equation}
	where $l(\Gamma)$ is the length of $\Gamma$, the the following facts hold true.
	\begin{enumerate}
		\item The curve $\Gamma$ is a subset of the resolvent set of $B$ enclosing a finite set $\widehat{\Lambda}$ of non-zero eigenvalues of $B$;
		\item The dimension of the range of $P$ is equal to the dimension of the range of $\widehat{P}$, where $\widehat{P} = \frac{1}{2\pi i} \int_{\Gamma} (\lambda I - B)^{-1} d \lambda$.
	\end{enumerate}
	\label{thm:rosasco}
\end{theorem}

From the theorem above, we can take $\Gamma$ as in \eqref{eq:def-gamma}, i.e. as the boundary of the rectangle
\begin{equation}
\big\{ \lambda \in \bbC \;|\; \frac{\lambda_K + \lambda_{K+1}}{2} \leq re(\lambda) \leq \|T\|_{\cH \to \cH} +1, |im(\lambda)|\leq 1 \big\}.
\label{eq:def-gamma2}
\end{equation}
For $\|E\|_{op} \leq \|E\|_{HS}$ small enough, $\Gamma$ contains exactly the top $K$ eigenvalues of $\tT$. Combining the three points up, we obtain sufficient conditions for the range of $v_1 + V_2Y_E$ to be the leading invariant subspace of $\tT$.
\begin{proposition}
	Assuming 
	\begin{gather}
		\|Y_E\|_{HS} < \min\Big\{1, \frac{1}{C_\cH} \Big\} \label{eq:conditionYE} \\
		\|E_{11}\|_{HS} + \|E_{12}\|_{HS} < \min\Big\{ \frac{\lambda_K-\lambda_{K+1}}{2}, 1 \Big\}, \label{eq:conditionE11E12}\\
		\|E\|_{HS} \leq \min\Big\{ \frac{\eta ^2}{\eta + l(\Gamma)/2\pi}, 1 \Big\}\label{eq:conditionE},
	\end{gather}
	where $\Gamma,\eta$ are defined according to \eqref{eq:def-gamma2} and \eqref{eq:def-eta} respectively, the range of $V_1+V_2 Y_E$ is the leading invariant subspace of $\tT$.
	\label{lm:step3main}
\end{proposition}
\begin{proof}
	First, our choice of $\Gamma$ contains and only contains the top $K$ eigenvalues of $T$. Next, since we assumed $\tT$ also has only real eigenvalues, Lemma \ref{lm:eigenoninvar} applies. So from \eqref{eq:NKsubspace-eigenbound}, \eqref{eq:conditionYE}, and \eqref{eq:conditionE11E12}, we know
	\begin{equation*}
	\sigma(T_{11} + E_{11} + E_{12}Y_E) \subset \big( \frac{\lambda_K + \lambda_{K+1}}{2},  \lambda_1 + 1 \big),
	\end{equation*}
	which is enclosed in $\Gamma$. We also see from \eqref{eq:conditionYE} that Lemma \ref{lm:kdimensional} applies. Finally, condition \eqref{eq:conditionE} on $E$ ensures Theorem \ref{thm:rosasco} applies, so $\tT$ has only $K$ eigenvalues in $\Gamma$. It thus follows the invariant subspace induced by $V_1+V_2 Y_E$ is the $K$-dimensional leading invariant subspace of $\tT$.
\end{proof}

\subsection{Step four: bound uniform consistency error}
In this step, we bound the uniform consistency error
\[
\inf\{\|V_1 - \tV_{1}Q\|_{2\to\infty}:Q\in\bbU^K\},
\]
where $\tV_1$ has orthonormal columns spanning the leading invariant subspace of $\tT$. Since we require $\tV_1^* \tV_1 = I$, we need to orthonormalize the ``columns'' of $V_1 + V_2 Y_E$. Let us define $Y_E^*$ to be the adjoint of $Y_E$ with respect to $\bbC^K$ and $l^2$. We can verify that for some $Q\in\bbU^K$, $\tV_1Q = (V_1 + V_2 Y_E)(I + Y_E^*Y_E)^{-1/2}$ because\footnote{As we will see from Lemma \ref{lm:yprops}, $(I + Y_E^*Y_E)^{-1/2}$ is well-defined when $\|Y_E\|_{HS}$ is small.} 
\[
(V_1 + V_2 Y_E)^*(V_1 + V_2 Y_E) = V_1^*V_1 + Y_E^*V_2^*V_2Y_E = I + Y_E^*Y_E.
\]

Meanwhile, note that by assumption, $\|\cdot\|_{\cH}$ is stronger than $\|\cdot\|_{\infty}$, i.e. $C_\cH \|f\|_{\cH} \geq \|f\|_{\infty}$ for all $ f \in \cH$. This implies for $\forall l \in \tl^2$
\[
C_\cH \|l\|_{\tl^2} = C_\cH \|V_2 l \|_{\cH} \geq \|V_2 l \|_{\infty} \geq \|V_2 l \|_{L^2} = \| l \|_{l^2}.
\]
The consequence of this is 
\begin{equation}
C_\cH \|Y_E\|_{2 \to \tl^2} \geq \|Y_E\|_{2 \to l^2}.
\label{eq:convertynorm}
\end{equation}

We also need the following handy result.
\begin{lemma}
	Let $Y: \bbC^K \rightarrow l_{2}$ have operator norm $\|Y\|_{2 \to l^2} < 1$. Then
	\begin{equation}
	\big\|(I + Y^*Y)^{-\frac{1}{2}} \big\|_2 \leq 1, \quad \big\|I - (I + Y^*Y)^{-\frac{1}{2}} \big\|_2 \leq \|Y\|_{2 \to l^2}.
	\end{equation}
	\label{lm:yprops}
\end{lemma}
\begin{proof}
	Suppose $\|Y\|_{2 \to l^2}= r < 1$. Note that $\| Y^*Y \|=r^2 < 1$, so the Hermitian matrix $I+Y^*Y$ is invertible with spectrum in $[1, 1+r^2]$. Consequently, $\sigma\big((I + Y^*Y)^{-\frac{1}{2}} \big) \subset [\frac{1}{\sqrt{1+r^2}}, 1]$, so $\big\|(I + Y^*Y^{-\frac{1}{2}}) \big\| \leq 1$.
	
	Similarly, we have $\sigma(I - (I + Y^*Y)^{-\frac{1}{2}}) \subset [0, 1 - \frac{1}{\sqrt{1+r^2}}]$. It remains to verify $1 - \frac{1}{\sqrt{1+r^2}} \leq r$, which is easy.
\end{proof}

Now we have
\begin{proposition}
	Suppose $\| Y_E \|_{2\to \tl^2 } < 1/C_\cH$ and  $\tV_1Q = (V_1 + V_2 Y_E)(I + Y_E^*Y_E)^{-1/2}$ for some $Q\in\bbU^K$. We have
	\[
	\inf\{\|V_1 - \tV_{1}Q\|_{2\to\infty}:Q\in\bbO^K\} \leq C_\cH (\|V_1\|_{2\to\infty} +1) \| Y_E \|_{2\to \tl^2 }.
	\]
	\label{prop:uniferror}
\end{proposition}
\begin{proof}
The condition $\| Y_E \|_{2\to \tl^2 } < 1/C_\cH$ ensures Lemma \ref{lm:yprops} applies. We then directly calculate
\begin{align*}
&\quad \inf \{\|V_1 - \tV_1 Q\|_{2\to\infty}:Q\in\bbO^K\} \\
&\leq \big\| V_1 - (V_1+V_2 Y_E) (I + Y_E^*Y_E)^{-\frac{1}{2}} \big\|_{2\to \infty} \\
&\leq \big\| V_1( I- (I + Y_E^*Y_E)^{-\frac{1}{2}}) \big\|_{2\to \infty} + \big\| V_2 Y_E (I + Y_E^*Y_E)^{-\frac{1}{2}} \big\|_{2\to \infty} \\
&\leq \| V_1 \|_{2\to\infty} \|Y_E\|_{2 \to l^2} + \| V_2 Y_E \|_{2\to \infty} \\
&\leq \|V_1\|_{2\to\infty} \|Y_E\|_{2 \to l^2} + C_\cH \| Y_E \|_{2\to \tl^2 } \\
&\leq C_\cH (\|V_1\|_{2\to\infty} +1) \| Y_E \|_{2\to \tl^2 }.
\end{align*}
\end{proof}  

\subsection{Step five: put all pieces together} We combine the previous steps together and prove Theorem \ref{thm:main-thm}. To this end, we need the following lemma that relates $\|E_{ij}\|_{HS}$, $i,j \in \{1,2\}$ to $\|E\|_{HS}$.
\begin{lemma}
	Let 
	\[
	C_3 = \max \Big \{C_\cH, 1 + C_\cH \|V_1\|_{2 \to \cH}, \|V_1\|_{2 \to \cH}(1 + C_\cH \|V_1\|_{2 \to \cH})  \Big \},
	\] 
	then for any $i,j \in \{1,2\}$
	\[
	\|E_{ij}\|_{HS} \leq C_3 \|E\|_{HS}.
	\]
	\label{lm:EijandE}
\end{lemma}
\begin{proof}
	Note the fact 
	\[
	\|V_{i}^* E V_j \|_{HS} \leq \|V_i^*\|_{op} \|E\|_{HS} \|V_j\|_{op}.
	\]
	We plug in the bounds from item 3 in Lemma \ref{lm:induced_hilbert} to obtain the stated result.
\end{proof}

\begin{proof}[Proof of Theorem \ref{thm:main-thm}]
Define $C_1$ as
\[
C_1 := \frac{1}{C_3} \min \Big\{ \frac{\lambda_{K}-\lambda_{K+1}}{8}, \frac{1}{2}, \frac{\delta}{4}, \frac{\delta}{4 C_\cH}, \frac{C_3\eta ^2}{\eta + l(\Gamma)/2\pi} \Big\}.
\]
We have $\|E\|_{HS} \leq C_1$ by assumption. By Lemma \ref{lm:EijandE}, this assumption implies 
\[
\|E_{ij}\|_{HS} \leq C_3 \|E\|_{HS} <  \frac{\delta}{4}
\]
for all $i,j \in \{1,2\}$. Thus
\begin{gather*}
	s_E = \delta - \|E_{11}\|_{HS} - \|E_{22}\|_{HS} \geq \frac{\delta}{2},\\
	\frac{\| E_{21} \|_{HS}   \| E_{12} \|_{HS}}{s_E^2} < \frac{1}{4}.
\end{gather*}
Proposition \ref{lm:nk} guarantees \eqref{eq:quad_eq} has a solution $Y_E$ and
\begin{equation}
\|Y_E\|_{HS} \leq \frac{4C_3}{\delta} \|E\|_{HS} < \min\{\frac{1}{C_\cH}, 1\}.
\label{eq:quad_eq_sol}
\end{equation}
We check that our choice of $C_1$ satisfies condition \eqref{eq:conditionE} and \eqref{eq:conditionE11E12} so Proposition \ref{lm:step3main} implies the invariant space from Proposition \ref{lm:nk} is the leading invariant space. Finally, \eqref{eq:quad_eq_sol} implies the conditions of Proposition \ref{prop:uniferror} are satisfied so we have 
\[
\inf \{\|V_1 - \tV_1 Q\|_{2\to\infty}:Q\in\bbU ^K\} \leq C_2 \|E\|_{HS}
\]
where $C_2 = 4C_3C_\cH (\|V_1\|_{2\to\infty} +1)/\delta$.

\end{proof}

\section{Application to normalized spectral clustering}
\label{sec:spectralClustering}

In spectral clustering, we start from a subset  $\cX \subset \bbR^p$, a probability measure $\bbP$ on $\cX$,\footnote{Assume the underlying $\sigma$-algebra of $\bbP$ is the Lebesgue $\sigma$-algebra.} and a continuous symmetric positive definite real-valued kernel function $k:\cX\times\cX\to\bbR$. After observing samples $X_1,\dots,X_n\overset{\iid}{\sim} \bbP$, we construct matrix $K_n\in\bbR^{n\times n}$ of their pairwise similarities: $K_n = \begin{bmatrix}\frac1nk(X_i,X_j)\end{bmatrix}_{i,j = 1}^n$, and then normalize it to obtain the {\it normalized Laplacian matrix}
\[
L_n = D_n^{-\frac12}K_nD_n^{-\frac12},
\]
where $d_n = K_n1_n$ and $D_n = \diag(d_n)$ is the degree matrix.\footnote{The normalized Laplacian matrix is usually defined as $I_n - L_n$, but the eigenvectors of $L_n$ and those of $I_n - L_n$ are identical and it is more convenient to study $L_n$.} It is possible to show that $L_n$ is symmetric and semi-positive definite, so it has an eigenvalue decomposition. We denote the eigenpairs of $L_n$ by $(\hlambda_k, v_k)$ and sort the eigenvalues in descending order: 
\[
\hlambda_1\ge\dots\ge\hlambda_n \ge 0.
\]
In this paper, we normalize the eigenvectors of $L_n$ so that $n^{-\frac12}\|v_k\|_2 = 1$. The spectral embedding matrix is $V\in\bbR^{n\times K}$ whose columns are $v_1,\dots,v_K$. 

Suppose for now that the kernel $k$ is bounded away from $0$ by a positive number and bounded from above. The operator counterpart of $L_n$ is the following operator, which can be shown to be a bounded linear operator in $\cL(C_b(\cX))$
\begin{equation}
(\hT_nf)(x) = \int_{\cX}\frac{k(x,y)}{d_n(x)^{1/2}d_n(y)^{1/2}}f(y)dP_n(y),
\label{eq:def-hT_n}
\end{equation}
where $d_n(x) = \int_{\cX}k(x,y)dP_n(y)$ is the sample degree function. Although $\hT_n$ is introduced as an operator in $\cL(C_b(\cX))$, we remark that the definitive element for $\hT_n$ is the integral form and the domain space and range space need not be restricted to $C_b(\cX)$. In fact, the actual $\hT_n$ we shall work with is an operator between Hilbert spaces; $\cL(C_b(\cX))$ is only chosen here for the ease of understanding. The same remark also applies to other operators we shall subsequently define. 

The operator $\hT_n$ is the operator counterpart of $L_n$ because $\rho_n\circ\hT_n = L_n\circ\rho_n$, where $\rho_n:C_b(\cX)\to \bbC^n$ is the restriction operator defined as
\[
\rho_nf = \begin{bmatrix}f(X_1) & \dots & f(X_n)\end{bmatrix}^T.
\] 
In other words, if we identify functions $f\in C_b(\cX)$ with vectors $v\in\bbC^n$ by the restriction operator $\rho_n$, $\hT_n$ ``behaves as'' $L_n$. The eigenvalues and eigenvectors(functions) of $\hT_n$ and $L_n$ are also closely related in the following sense.

\begin{lemma}
	\label{lm:spectralEquivalence}
	Suppose real-valued kernel function $k(x,y)$ is continuous and bounded from below and above: $0 < \kappa_l < k(x,y) < \kappa_u < \infty$. Let $\hT_n$ be defined as in \eqref{eq:def-hT_n} where the domain space and range space are both $C_b(\cX)$. If $(\hlambda,f)$ is a non-trivial eigenpair of $\hT_n$ (\ie\ $\hlambda \ne 0$), then $(\hlambda,\rho_nf)$ is an eigenpair of $L_n$. Conversely, if $(\hlambda,v)$ is an eigenpair of $L_n$, then $(\hlambda,\hf)$, where
	\begin{equation}
	\hf(x) = \frac{1}{\hlambda n}\sum_{i=1}^n\frac{k(x,X_i)}{d_n(x)^{1/2}d_n(X_i)^{1/2}}v_i,
	\label{eq:nystrom-ext}
	\end{equation}
	is an eigenpair of $\hT_n$ with $\hf \in C_b(\cX)$. Moreover, this choice of $\hf$ is such that $\|\hf\|_{L^2(\cX, \bbP_n)} = 1$ and the restriction of $\hf$ onto sample points agrees with $v$, i.e. $\rho_n \hf = v$.
\end{lemma}

\begin{proof}
	Let $(\hlambda,\hf)$ be an eigenpair of $\hT_n$: $\hT_n\hf = \hlambda\hf$. We check that $(\hlambda,\rho_n\hf)$ is an eigenpair of $L_n$:
	\[
	L_n\rho_n\hf = \rho_n\hT_n\hf = \rho_n\hlambda\hf = \hlambda\rho_n\hf.
	\]
	Conversely, if $(\hlambda,v)$ is an eigenpair of $L_n$, we check that $(\hlambda,\hf)$ is an eigenpair of $\hT_n$:
	\[
	\begin{aligned}
	(\hT_n\hf)(x) &= \frac1n\sum_{i=1}^n\frac{k(x,X_i)}{d_n(x)^{1/2}d_n(X_i)^{1/2}}\left\{\textstyle\frac{1}{\hlambda n}\sum_{j=1}^n\frac{k(X_i,X_j)}{d_n(X_i)^{1/2}d_n(X_j)^{1/2}}v_j\right\} \\
	&= \frac1n\sum_{i=1}^n\frac{k(x,X_i)}{d_n(x)^{1/2}d_n(X_i)^{1/2}}\left\{\textstyle\frac{1}{\hlambda}[L_nv]_i\right\} \\
	&= \frac1n\sum_{j=1}^n\frac{k(x,X_j)}{d_n(x)^{1/2}d_n(X_j)^{1/2}}v_i \\
	&= (\hlambda\hf)(x).
	\end{aligned}
	\]
	It remains to check $\hf$ is indeed in $C_b(\cX)$. To this end, note that since the kernel function $k$ is continuous and bounded from above, we know $k(x,X_j) \in C_b(\cX)$. Since $k$ is bounded from below, we know $d_n(x)$ is continuous and $d_n(x) > \kappa_l$, so $k(x,X_j)/(d_n(x)^{1/2}d_n(X_i)^{1/2}) \in C_b(\cX)$. Thus the average of such terms $\hf$ is also in $C_b(\cX)$.
\end{proof}

The population version of $\hT_n$ is the {\it normalized Laplacian operator} 
\[
T:C_b(\cX)\to C_b(\cX)\text{ defined as }(Tf)(x) = \int_{\cX}\frac{k(x,y)}{d(x)^{1/2}d(y)^{1/2}}f(y)dP(y),
\]
where $d(x) = \int_{\cX}k(x,y)dP(y)$ is the (population) degree function. Under appropriate assumptions, it can be shown that we can choose $\{f_i\}_{i=1}^K$, the top $K$ eigenfunctions of $T$, to be real-valued and orthonormal in $L^2(\cX, P)$. We can thus define $V_1:\bbC^K \to C_b(X)$ as $V_1 \alpha = \sum_{i=1}^K \alpha_i f_i$. We can similarly define $\hV_1$ with $\{\hf_k\}_{k=1}^K$, the extension of top $K$ eigenvectors of $L_n$ according to \eqref{eq:nystrom-ext}. Our goal in this section is to apply our general theory to prove the following result.
\begin{theorem}
	Under the general assumptions defined below, there exists $C_4,C_5$ that are determined by $\cX, \bbP, k$ such that whenever sample size $n \geq C_4 \tau$ for some $\tau >1$, we have with confidence $1-8e^{-\tau}$
	\[
	\inf \{\|V_1 - \hV_1 Q\|_{2\to\infty}:Q\in\bbU^K\} \leq C_5 \frac{\sqrt{\tau}}{\sqrt{n}}.
	\]
	\label{thm:spec-consis}
\end{theorem}
 The general assumptions referred to in Theorem \ref{thm:spec-consis} are
\begin{ntheorem}
	The set $\cX$ is a bounded connected open set in $\bbR^p$ with a nice boundary.\footnote{We need the boundary to be quasi-resolved \cite{burenkov1998} for inequality \eqref{eq:burenkov} and $C^{\infty}$ for lemma \ref{lm:coveringnumber} \cite{edmunds1996, cucker2002}. We also need $\cX$ to satisfy the cone condition \cite{burenkov1998} We omit the definitions of these conditions because the precise definitions are very technical and not relevant to the main story of the paper.}. The probability measure $\bbP$ is defined with respect to Lebesgue measure and admits a density function $p(x)$. Moreover, there exists constants $0<p_l < p_u<\infty$ such that $p_l < p(x) < p_u$ almost surely with respect to the Lebesgue measure. The kernel $k(\cdot,\cdot) \in C_b^{p+2}(\cX \times \cX)$ is symmetric, positive, and there exists constants $0<\kappa_l < \kappa_u<\infty$ such that $\kappa_l < k(x,y) < \kappa_u $ for $\forall x,y \in \cX$. Treated as an operator from $L^2(\cX, \bbP)$ to $L^2(\cX, \bbP)$, the eigenvalues of $T$ satisfy $\lambda_1 \geq \ldots \geq \lambda_K > \lambda_{K+1} \geq \ldots \geq 0$. The top $K$ eigenfunctions of $T$, $\{f_i\}_{i=1}^K \subset C_b^{p+2}(\cX)$. \footnote{Function space $C_b^{p+2}(\cX)$ shall be defined in section \ref{sec:sob}.}
\end{ntheorem}

\subsection{Overview of the proof}
The most challenging parts in applying the general theory are to identify the correct Hilbert space $\cH$ to work with, and to show that $T-\hT_n$, as an operator from $\cH$ to $\cH$, has Hilbert-Schmidt norm tending to zero as $n$ goes to infinity. It turns out under the general assumptions, we may set $\cH$ to be a Sobolev space of sufficiently high degrees. As for bounding $\|T-\hT_n\|_{HS}$, we first decompose $T,\hT_n$ as the product of three operators. Let us define  
\begin{gather}
	D^{-1/2}:\cH\to \cH\text{ as }(D^{-1/2}f)(x) = f(x)/\sqrt{d(x)}, \\
	K:\cH\to \cH\text{ as }(Kf)(x) = \int_{\cX}k(x,y)f(y)dP(y).
	\label{eq:def-ofK}
\end{gather}
Then $T=D^{-1/2} K D^{-1/2}$. Similarly, we have $\hT_n = D_n^{-1/2} K_n D_n^{-1/2}$ where $D_n^{-1/2},K_n$ are the sample level version of $D^{-1/2}$ and $K$ defined using $d_n$ and $\bbP_n$. We shall establish the concentration of $K_n$ to $K$ and $D_n^{-1/2}$ to $D^{-1/2}$ and invoke triangular inequality to bound $\|T-\hT_n\|_{HS}$.

Despite that the general theory does all the heavy lifting, there is one additional step we must take to finish the full proof of Theorem \ref{thm:spec-consis}. In our general theory, $\tV_1$ has columns orthonormal in $L^2(\cX,\bbP)$ that span the leading invariant space of $\tT$. In theorem \ref{thm:spec-consis} however, the same leading invariant space is spanned by the columns of $\hV_1$, which are only orthonormal in $L^2(\cX,\bbP_n)$. Morally speaking, when $n$ is large, $\hV_1$ and $\tV_1$ are roughly the same up to some unitary transformation, so switching from $\tV_1$ to $\hV_1$ should not inflate the consistency error by any order of magnitude. The exact error bound shall be obtained through some uniform law of large numbers.

The rigorous treatment shall be presented in five parts. In part one, we introduce the Sobolev space we work with and lay out its basic properties. In part two, we bound the norm of operator differences such as $\| D_n^{-1/2} -D^{-1/2} \|_{\cH \to \cH}$ and $\|K - K_n \|_{HS}$ and express $\|T-\hT_n\|_{HS}$ in terms of them. In part three, we invoke concentration results in Hilbert spaces and relate the norm of operator differences to sample size. In part four, we check the remaining conditions required by our general theory and combine all previous pieces together. In part five, we deal with the error induced by the difference of $\hV_1$ and $\tV_1$ and complete the proof.

\subsection{Part one: The Sobolev space $\cH^s$}
\label{sec:sob}
First recall that by assumption $\cX$ is a bounded connected open subset of $\bbR^p$ with a nice boundary. Given $s \in \bbN$, the Sobolev space $\cH^s=\cH^s(\cX)$ of order $s$ is defined as
\[
\cH^s\;:=\; \{ f \in L^2(\cX, dx) \;|\; D^\alpha f \in L^2(\cX, dx), \forall |\alpha|=s \},
\]
where $D^\alpha f$ is the (weak) derivative of $f$ with respect to the multi-index $\alpha$ and $L^2(\cX, dx)$ is the complex Hilbert space of complex-valued functions square integrable under Lebesgue measure. The space $\cH^s$ is a separable Hilbert space with respect to the inner product
\[
\ip{f,g}_{\cH^s} = \ip{f,g}_{L^2(\cX, dx)} + \sum_{|\alpha|=s} \ip{D^\alpha f,D^\alpha g}_{L^2(\cX, dx)}.
\]

Let $C_b^s(\cX)$ be the set of complex-valued continuous bounded functions such that all the derivatives up to order $s$ exist and are continuous bounded functions. The space $C_b^s(\cX)$ is a Banach space with respect to the norm
\[
\|f\|_{C_b^s} = \|f\|_\infty + \sum_{|\alpha|=s} \| D^\alpha f \|_\infty.
\]

Since $\cX$ is bounded, we know $C_b^s(\cX) \subset \cH^s$ and $\|f\|_{\cH^s} \leq C_{s} \|f\|_{C_b^s}$ where $C_s$ is a constant only depending on $s$. We also know from the Sobolev embedding theorem (see Chapter 4.6 of \citet{burenkov1998}) that for $l,m \in \bbN$ with $l-m > p/2$, we have 
\begin{equation}
\cH^l \subset C_b^m(\cX) \qquad \|f\|_{C_b^m} \leq C_{m,l} \|f \|_{\cH^l}
\label{eq:sobo_embedding}
\end{equation}
where $C_{m,l}$ is a constant depending only on $m$ and $l$.

Taking $l=s=\floor{p/2} + 1$ and $m=0$, we see
\[
C_b^s(\cX) \subset \cH^s \subset C_b(\cX),
\]
with $ \|f\|_\infty \leq C_6 \|f \|_{\cH^s}$ for $\forall f \in \cH^s$ for some constant $C_6$. This norm relationship suggests that $\cH^s$ is a RKHS with a bounded kernel $s(\cdot,\cdot)$.

\subsection{Part two: bounds on operator differences}
Similar to \eqref{eq:def-ofK}, we define multiplication operators
\begin{gather}
D^{1/2}:\cH^s\to \cH^s\text{ as }(D^{1/2}f)(x) = \sqrt{d(x)}f(x), \\
D:\cH^s\to \cH^s\text{ as }(Df)(x) = d(x)f(x).
\end{gather}
In this subsection, we show $D^{1/2}, D^{-1/2}, D, D_n^{1/2}, D_n^{-1/2}, D_n \in \cL(\cH^s)$ and their operator norms are appropriately bounded, that is
\begin{lemma}
	Under the general assumptions, all the following operators are bounded linear operators in $\cL(\cH^s)$, and there exists a suitable constant $C_7 > 0$ such that
	\begin{gather}
	\|D^{1/2}\|_{\cH^s \to \cH^s},\|D^{-1/2}\|_{\cH^s\to\cH^s},\|D_n^{1/2}\|_{\cH^s\to\cH^s},\|D_n^{-1/2}\|_{\cH^s\to\cH^s} \leq C_7 \\
	\|(D^{1/2}+D_n^{1/2})^{-1}\|_{\cH^s\to\cH^s} \leq C_7, \quad \|D-D_n\|_{\cH^s\to\cH^s} \leq  C_7 \|d-d_n\|_{\cH^{p+2}}
	\end{gather}
	\label{lm:boundonD}
\end{lemma}

\begin{proof}
	Let $C_8=\|k\|_{C_b^{p+2}(\cX \times \cX)}$. For any $x \in \cX$, clearly $k_x:=k(\cdot, x) \in C_b^{p+2}(\cX)$ with $\|k_x \|_{C_b^{p+2}} \leq C_8$. Since $d$ and $d_n$ are some weighted average of average of $k_x$, it follows
	\[
	\|d \|_{C_b^{p+2}},  \|d_n \|_{C_b^{p+2}} \leq C_8. 
	\]
	Since $d,d_n$ inherit the $\kappa_l,\kappa_u$ pointwise bound from $k(\cdot,\cdot)$, we know $d^{1/2},d_n^{1/2},d^{-1/2},d_n^{-1/2} \in C_b^{p+2}(\cX)$ with
	\[
	\|d^{1/2} \|_{C_b^{p+2}},  \|d_n^{1/2} \|_{C_b^{p+2}}, \|d^{-1/2} \|_{C_b^{p+2}},  \|d_n^{-1/2} \|_{C_b^{p+2}} \leq C_9.
	\]
	Next, we know from Lemma 15 of Chapter 4 of \citet{burenkov1998} that for $g \in C_b^{s}(\cX)$ and $f \in \cH^s$, we have $gf \in \cH^s$ and
	\begin{equation}
	\|gf\|_{\cH^s} \leq \|g\|_{C_b^s} \|f\|_{\cH^s}.
	\label{eq:burenkov}
	\end{equation}
	We can use this inequality to prove $D^{1/2}, D^{-1/2}, D, D_n^{1/2}, D_n^{-1/2}, D_n \in \cL(\cH^s)$ and bound their operator norm. For example, plugging in $g=d^{1/2},d_n^{1/2},d^{-1/2},d_n^{-1/2}$ into \eqref{eq:burenkov}, and noticing for these choices of $g$, $\|g\|_{C_b^s} \leq \|g\|_{C_b^{p+2}}$ because $p+2>s$, we conclude
	\[
	\|D^{1/2}\|,\|D^{-1/2}\|,\|D_n^{1/2}\|,\|D_n^{-1/2}\| \leq C_9.
	\]
	Note that by the embedding theorem, $\cH^{p+2}$ can be embedded into $C_b^s(\cX)$, so plugging in $d-d_n \in C_b^{p+2}(\cX) \subset \cH^{p+2}$, we see
	\[
	\|D-D_n\| \leq \|d-d_n\|_{C_b^s} \leq C_{10} \|d-d_n\|_{\cH^{p+2}}.
	\]
	For the bound on $\|(D^{1/2}+D_n^{1/2})^{-1}\|$, we follow essentially the same route. We first bound $\|d^{1/2}+d_n^{1/2} \|_{C_b^{p+2}}$, then argue $d^{1/2}+d_n^{1/2}$ has pointwise lower and upper bound. It then follows that $\|(d^{1/2}+d_n^{1/2})^{-1} \|_{C_b^{p+2}} \leq C_{11}$, and we see via \eqref{eq:burenkov} that $\|(D^{1/2}+D_n^{1/2})^{-1}\|_{\cH^s,\cH^s} \leq C_{9}$. Taking $C_7$ as the maximum of $C_9$ to $C_{11}$ completes the proof.
\end{proof}
\begin{lemma}
	Under the general assumptions, we have
	\[
	\|T-\hT_n\|_{HS} \leq C_{12} \Big( \big(  \|K_n\|_{HS} + \|K\|_{HS} \big) \|d-d_n\|_{\cH^{p+2}} + \|K-K_n\|_{HS} \Big).
	\]
	\label{lm:decompTTn}
\end{lemma}
\begin{proof}
	First note
	\begin{align*}
	&\quad D^{-1/2} - D_n^{-1/2} \\
	&=D_n^{-1/2}(D_n^{1/2} - D^{1/2})D^{-1/2}\\
	&=D_n^{-1/2}(D_n - D)(D_n^{1/2} + D^{1/2})^{-1}D^{-1/2}.
	\end{align*}
	Applying the bounds from Lemma \ref{lm:boundonD}, we see 
	\[
	\|D^{-1/2} - D_n^{-1/2}\| \leq C_7^4 \|d-d_n\|_{\cH^{p+2}}.
	\]
	We also have decomposition
	\begin{align*}
	&\quad D^{-1/2}KD^{-1/2} - D_n^{-1/2}K_nD_n^{-1/2} \\
	& = D^{-1/2}K(D^{-1/2} - D_n^{-1/2}) + (D^{-1/2}K - D_n^{-1/2}K_n)D_n^{-1/2} \\
	& = D^{-1/2}K(D^{-1/2} - D_n^{-1/2}) + D^{-1/2}(K-K_n)D_n^{-1/2} + (D^{-1/2} - D_n^{-1/2})K_nD_n^{-1/2}.	
	\end{align*}
	Taking Hilbert-Schmidt norm on both sides, we have\footnote{We haven't shown $K,K_n$ are Hilbert-Schmidt operators from $H^s$ to $H^s$ yet. This is shown in Lemma \ref{lm:Kd-concentration}}
	\begin{align*}
	& \quad \|T-\hT_n\|_{HS} \\
	& \leq C_7^5 \|K\|_{HS} \|d-d_n\|_{\cH^{p+2}} + C_7^2 \|K-K_n\|_{HS} + C_7^5 \|K_n\|_{HS} \|d-d_n\|_{\cH^{p+2}} \\
	& \leq C_{12} \Big( \big(  \|K_n\|_{HS} + \|K\|_{HS} \big)\|d-d_n\|_{\cH^{p+2}} + \|K-K_n\|_{HS} \Big)
	\end{align*}
	for some appropriately chosen $C_{12}$.
\end{proof}

\subsection{Concentration on Hilbert Space}
In this subsection, we show $K,K_n$ are both Hilbert-Schmidt operator from $\cH^s$ to $\cH^s$ and establish some concentration results regarding $\|d-d_n\|_{\cH^{p+2}}$ and $\|K-K_n\|_{HS}$. With these results and Lemma \ref{lm:decompTTn}, we will be able to bound $\|T-\hT_n\|_{HS}$. The required concentration bounds are obtained through the following Theorem on the concentration in (complex) Hilbert space (see section 2.4 of \citet{rosasco2010Learninga}).

\begin{lemma}
	Let $\xi_1,\ldots,\xi_n$ be zero mean independent random variables with values in a separable (complex) Hilbert space $\cH$ such that $\| \xi_i \|_\cH \leq C$ for all $i \in [n]$. Then with probability at least $1-2e^{-\tau}$, we have
	\[
	\big\| \frac{1}{n}\sum_{i=1}^n \xi_i \big\|_\cH \leq \frac{C \sqrt{2\tau}}{\sqrt{n}}.
	\]
	\label{lm:Hilbert-concentration}
\end{lemma}

With this lemma, we can show
\begin{lemma}
	Under the general assumptions, the following facts hold true:
	\begin{enumerate}
		\item For some constant $C_{13}$, with confidence $1-2e^{-\tau}$
		\[
		\|d-d_n\|_{\cH^{p+2}} \leq C_{13} \frac{\sqrt{\tau}}{\sqrt{n}}.
		\]
		\item Both $K$ and $K_n$ are Hilbert-Schmidt operators from $\cH^s$ to $\cH^s$, and there exists some constant $C_{14}$ that doesn't depend on $n$ such that their Hilbert-Schmidt norm $\|K_n\|_{HS},\|K\|_{HS} \leq C_{14}$ is bounded.
		\item For some constant $C_{15}$, with confidence $1-2e^{-\tau}$
		\[
		\|K-K_n\|_{HS} \leq C \frac{\sqrt{\tau}}{\sqrt{n}}.
		\]
	\end{enumerate}
	\label{lm:Kd-concentration}
\end{lemma}

\begin{proof}
	For item 1, consider random variables $\xi_i = k(\cdot,X_i) - d \in \cH^{p+2}$ for $i \in [n]$. They are clearly zero mean. From the proof of Lemma \ref{lm:boundonD}, we see $d, k(\cdot,X_i) \in C_b^{p+2}(\cX)$. We thus have
	\[
	\|\xi_i\|_{\cH^{p+2}} \leq \|k(\cdot,X_i)\|_{\cH^{p+2}} + \|d\|_{\cH^{p+2}} \leq C_\cX \big( \|k(\cdot,X_i)\|_{C_b^{p+2}} + \|d\|_{C_b^{p+2}} \big)
	\] 
	where $C_\cX$ is some constant depending on the Lebesgue measure of the bounded set $\cX$. This suggests $\xi_i$'s are bounded. Since $\cH^{p+2}$ is a separable Hilbert space, we apply Lemma \ref{lm:Hilbert-concentration} and conclude that we have with probability $1-2e^{-\tau}$
	\[
	\|d-d_n\|_{\cH^{p+2}}\leq C_{13} \frac{\sqrt{\tau}}{\sqrt{n}}.
	\]

	For item 2, let us fix any $x\in \cX$ and consider the operator $\ip{\cdot, s_{x}}_{\cH^s}k_{x}$ where $s_{X_i}:=s(\cdot, X_i)$. This operator is in fact a Hilbert-Schmidt operator from $\cH^s$ to $\cH^s$. To see this, note that $\|\ip{\cdot,s_{x}}_{\cH^s}k_{x}\|_{HS} = \|s_{x}\|_{\cH^s}\|k_{x}\|_{\cH^s}$. With the same reasoning used for proving item 1, we see $\|k_{x}\|_{\cH^s}$ has a bound uniform on $\forall x\in \cX$. It remains to show $\|s_{x}\|_{\cH^s}$ has a uniform bound. Let $\delta_{x}:\cH^s\to \bbC$ be the evaluation functional, i.e. $\delta_x(f)=f(x)$. We know from the embedding theorem that $\|\delta_{x}\|_{op} \leq C_6$ for all $x\in \cX$. But $s_x$ also induces this point evaluation functional, so by Riesz representation theorem, $\|s_x\|_{\cH^s} = \|\delta_{x}\|_{op} \leq C_6$. Hence for some $C_{14}$, $\|\ip{\cdot,s_{x}}_{\cH^s}k_{x}\|_{HS} \leq C_{14}$ for all $x \in \cX$. Now let $x$ be random. We see $\|K\|_{HS} = \|\bbE \ip{\cdot, s_{X_i}}_{\cH^s}k_{X_i}\|_{HS} \leq \bbE \| \ip{\cdot, s_{X_i}}_{\cH^s}k_{X_i}\|_{HS} \leq C_{14}$, i.e.  $K$ is Hilbert-Schmidt. By the same reasoning, we see the claim for $K_n$ in item 2 is also true. 
	
	For item 3, consider random variables $\omega_i:=\ip{\cdot, s_{X_i}}_{\cH^s}k_{X_i} - K$. We know from item 2 that $\omega_i \in HS(\cH^s)$. Since $\cH^s$ is separable, the Hilbert space $HS(\cH^s)$ is also separable. We also know $\omega_i$ is zero mean and $\|\omega_i\|_{HS} \leq 2C_{14}$ is bounded. We can thus apply Lemma \ref{lm:Hilbert-concentration} and conclude that we have with probability $1-2e^{-\tau}$
	\[
	\|K-K_n\|_{HS} \leq C_{15} \frac{\sqrt{\tau}}{\sqrt{n}}
	\]
	for $C_{15}:=2C_{14}$.
\end{proof}

Combining Lemma \ref{lm:decompTTn} and \ref{lm:Kd-concentration}, we obtain the result we want
\begin{proposition}
	Under the general assumptions, with probability $1-4e^{-\tau}$, we have
	\[
	\|T-\hT_n\|_{HS} \leq C_{16} \frac{\sqrt{\tau}}{\sqrt{n}}
	\]
	for some constant $C_{16}$.
	\label{prop:T-concen}
\end{proposition}
\begin{proof}
	A union bound and a direct application of lemma \ref{lm:decompTTn} will suffice for the proof.
\end{proof}

\subsection{Checking conditions for general theory}
In the first three paragraphs of section 3, we have laid out the conditions that must be satisfied for our general theory to apply. We've already checked most of them implicitly in the previous three subsections, but for completeness, we summarize all such conditions here and prove them.

\begin{lemma}
	Under the general conditions, the following facts hold true:
	\begin{enumerate}
		\item The Sobolev space $H^s$ is a subspace of $L^2(\cX, \bbP)$.
		\item The $\cH^s$ norm $\|\cdot \|_{\cH^s}$ is stronger than infinity norm.
		\item Both $T,\hT_n$ are Hilbert-Schmidt from $\cH_s$ to $\cH_s$.
		\item All eigenvalues of $T$ (counting multiplicity) can be arranged in a decreasing (possibly infinite) sequence of non-negative real numbers $\lambda_1 \geq \lambda_2 \geq \ldots \geq \lambda_K > \lambda_{K+1} \geq \ldots \geq 0$ with a positive gap between $\lambda_K$ and $\lambda_{K+1}$.
		\item The top $K$ eigenfunctions $\{f_i\}_{i=1}^K \subset \cH^s$ and can be picked to form an orthonormal set of functions in $L^2(\cX, \bbP)$.
		\item $\hT_n$ has a sequence of non-increasing, real, non-negative eigenvalues.
	\end{enumerate}
\label{lm:check-for-gen-theory}
\end{lemma}
\begin{proof}
	For item 1, this is because $H^s$ is a subspace of $L^2(\cX, dx)$ and under our assumptions on $\cX$ and $\bbP$, $L^2(\cX, dx)$ and $L^2(\cX, \bbP)$ are the same space. First of all, since the underlying $\sigma$-algebra of $\bbP$ is the Lebesgue $\sigma$-algebra, the set of measurable functions are the same. If $f \in L^2(\cX, dx)$, then $f$ is also in $L^2(\cX, \bbP)$ because $\int_\cX f\bar{f} dP = \int_\cX f\bar{f} p(x) dx \leq p_u \int_\cX f\bar{f} dx < \infty$. The converse is also true. It is not hard to see our assumptions ensure that the Lebesgue measure is absolutely continuous with respect to $\bbP$ with the density being $1/p(x)$ a.s.. Noticing $1/p(x) < 1/p_l$, we can prove the converse.
	
	Item 2 is the consequence of Sobolev embedding theorem and has been used time and again in the previous subsections. Item 3 is the joint consequence of Lemma \ref{lm:boundonD} and \ref{lm:Kd-concentration}. 
	
	For item 4 and 5, we first show $T$ as an operator from $L^2(\cX, \bbP)$ to $L^2(\cX, \bbP)$ is positive, self-adjoint, and Hilbert-Schmidt. Let $h(x,y):=k(x,y)/\sqrt{d(x)d(y)}$ denote the normalized kernel. The self-adjointness is due to the (conjugate) symmetry $h(\cdot,\cdot)$ inherited from $k(\cdot,\cdot)$:
	\begin{gather*}
		\ip{f, Tg} = \int f(x) \Big( \int \average[2.5]{h(x,y)}  \average[2.5]{g(y)} dP(y) \Big)dP(x)
		= \iint h(y,x)f(x)\average[2.5]{g(y)}dP(x)dP(y), \\
		\ip{Tf, g} = \int  \Big( h(x,y) f(y) dP(y) \Big) \average[2.5]{g(x)} dP(x) = \iint h(x,y)f(y)\average[2.5]{g(x)}dP(x)dP(y).
	\end{gather*}
	We thus see $\ip{f, Tg} = \ip{Tf, g}$, i.e. $T$ is self-adjoint. To see why {$T$ is Hilbert-Schmidt}, let real-valued functions $\{e_i\}_{i=1}^\infty$ be an orthonormal basis of $L^2(\cX, \bbP)$ . We calculate 
	\begin{gather*}
	\|T\|_{HS}^2 = \sum_{i=1}^{\infty}  \Bigg( \int \Big( \int h(x,y) e_i(y) dP(y) \Big)^2 dP(x) \Bigg) = \int \Big( \sum_{i=1}^{\infty} \ip{h_x, e_i}_{L^2}^2 \Big) dP(x) \\
	= \iint h^2(x,y) dP(x)dP(y) \leq \kappa_u^2/\kappa_l^2 < \infty.
	\end{gather*}
	The positive part is slightly more involved. To show $T$ is positive, we need to show for $\forall f \in L^2(\cX, \bbP)$
	\[
	\ip{Tf,f}=\iint k(x,y) \average[2.5]{f(x)} f(y) dP(x)dP(y) \geq 0.
	\]
	Let us fix a sample size $n$ and draw i.i.d. samples $X_1,X_2,\ldots,X_n \sim \bbP$. Then since the kernel $k(\cdot,\cdot)$ is positive definite, the quadratic form
	\[
	\frac{1}{n^2}\sum_{i,j=1}^n k(X_i,X_j)\average[2.5]{f(X_i)} f(X_j)
	\]
	is non-negative regardless of what samples we draw. It thus follows that the expectation of this quadratic form is non-negative. A simple calculation suggests that the expectation is in fact
	\[
	\frac{n-1}{n} \ip{Tf,f} + \frac{1}{n} \int k(x,x) \average[2.5]{f(x)} f(x) dP(x).
	\]
	
	Since by our assumption $k(x,x) \leq \kappa_u$ and $f\in L^2(\cX, \bbP)$, we see $\int k(x,x) \average[2.5]{f(x)} f(x) dP(x)$ is finite. Since $n$ can be arbitrarily large, $\ip{Tf,f}$ must be non-negative.

	According to the spectral theory for positive, self-adjoint, Hilbert-Schmidt operators we introduced in Section \ref{sec:spec-theory}, we immediately see most parts of item 4 and 5 are true. The remaining part to check for item 5 is that $\{f_i\}_{i=1}^K \subset \cH^s$, which is implied by our assumption that $\{f_i\}_{i=1}^K \subset C_b^{p+2}(\cX)$. The eigengap part in item 4 is also assumed by the general assumption. A nuance in item 4 is that the eigenvalues and eigenvectors there are under the premise that $T$ is an operator from $\cH^s$ to $\cH^s$. But since $\cH^s$ is a subspace of $L^2(\cX, \bbP)$, we can only have fewer eigenvalues than treating $T$ as in $\cL(L^2(\cX, \bbP))$. Plus, since our general assumptions imply $\{f_i\}_{i=1}^K \subset \cH^s$, the leading eigenspace remains unchanged after the restriction from $\cL(L^2(\cX, \bbP))$ to $\cH^s$.

	For item 6, this is true because of the relationship between the spectrum of $\hT_n$ and that of the symmetric positive semi-definite kernel matrix $L_n$. A sort of Lemma \ref{lm:spectralEquivalence} is also true with the $C_b(\cX)$ therein replaced by $\cH^s$

	\iffalse
	For item 1, because in the general assumptions we assumed
	\[
	\iint k(x,y)f(x)f(y)dP(x)dP(y) \geq 0
	\]
	for any $f \in L^2(\cX, \bbP)$. Plugging in $f/d^{1/2}$, we see $T$ is positive. 
	\fi
	
\end{proof}

Because of Lemma \ref{lm:check-for-gen-theory}, we can apply a slightly modified version Theorem \ref{thm:main-thm} (see the proof for Theorem \ref{thm:main-thm}) to obtain the following. 
\begin{proposition}
	For some constant $C_{17},C_{18}$, whenever $n>C_{17}\tau$, we have with confidence $1-4e^{-\tau}$ that
	\begin{equation}
		\|V_1 - (V_1+V_2Y)(I+Y^*Y)^{-1/2}\|_{2 \to \infty} \leq C_{18} \frac{\sqrt{\tau}}{\sqrt{n}}
	\end{equation}
	\label{prop:sampling-error}
\end{proposition}
\begin{proof}
	By proposition \ref{prop:T-concen}, we have with probability $1-4e^{-\tau}$,
	\[
	\|T-\hT_n\|_{HS} \leq C_{16} \frac{\sqrt{\tau}}{\sqrt{n}}.
	\]
	We can set $C_{17}$ sufficiently large that $C_{16} \frac{\sqrt{\tau}}{\sqrt{n}} \leq C_{16}/\sqrt{C_{17}}\leq C_1$, where $C_1$ is the constant from Theorem \ref{thm:main-thm}. Hence by an intermediate step in the proof of Theorem \ref{thm:main-thm}, we conclude
	\[
	\|V_1 - (V_1+V_2Y)(I+Y^*Y)^{-1/2}\|_{2 \to \infty} \leq C_{18} \frac{\sqrt{\tau}}{\sqrt{n}}.
	\]
\end{proof}

\subsection{Part five: error induced by $\hV_1$}
We deal with the error induced by operator $\hV_1$ not having orthonormal columns in $L^2(\cX, \bbP)$. Introduce the shorthand $W:=(V_1+V_2Y)(I+Y^*Y)^{-1/2}$ and we have
\begin{align}
\|V_1 - \hV_1 Q \|_{2 \to \infty} &\leq \|V_1 - W \|_{2 \to \infty} + \|W - \hV_1 Q \|_{2 \to \infty} \nonumber \\
&= \|V_1 - W \|_{2 \to \infty} + \|W - WW^*\hV_1 Q \|_{2 \to \infty} \nonumber \\
&\leq \|V_1 - W \|_{2 \to \infty} + \|W \|_{2 \to \infty} \| Q^T -W^*\hV_1 \|_2. \label{eq:QandWV}
\end{align}
Here, the equality in the second step is true because $\hV_1$ and $W$ span the same leading eigenspace\footnote{Since $\hV_1$ is constructed from the eigenvectors of $L_n$ which are linearly independent, the columns cannot be linearly dependent functions in $\cH^s\subset C_b(\cX)$.} and $W$ has orthonormal columns in $L^2(\cX, \bbP)$. Inspecting \eqref{eq:QandWV}, we see $\|V_1 - W \|_{2 \to \infty}$ is bounded by Proposition \ref{prop:sampling-error}, and $\|W \|_{2 \to \infty}$ is roughly $\|V_1 \|_{2 \to \infty}$ thus bounded, so it all boils down to how ``far'' $W^*\hV_1$ is from an unitary matrix in $\bbC^{K \times K}$. In fact, we have
\begin{lemma}
	Assume all the singular values of $W^*\hV_1$ are less than $2$, then there exists unitary matrix $Q\in \bbC^{K \times K}$ such that
	\[
	\|Q-W^*\hV_1\|_2 \leq 2 \|W\|_{2 \to H^s}^2 \sup_{{\|g\|_{\cH^s}\leq 1}} \Big|P_ng^2 - Pg^2 \Big|.
	\]
	\label{lm:bound-unitary}
\end{lemma}
\begin{proof}
	Suppose $W^*\hV_1$ admits singular value decomposition $W^*\hV_1 = A \Sigma B^*$, then $\Sigma = A^*W^* \hV_1 B$. Let $g_i:=WAe_i$ be the $i$-th column in $WA$, $\tg_i:=\hV_1 B e_i$ be the $i$-th column in $\hV_1 B$ where $\{e_i\}_{i=1}^K$ is the standard basis in $\bbR^K$. We know $\Sigma = (\Sigma_{ij})$ where $\Sigma_{ij} = \ip{\tg_i,g_j}_{L^2(\cX, \bbP)}$.
	
	Since $A,B$ are unitary matrices, $\{g_i\}_{i=1}^K$ are orthonormal in $L^2(\cX, \bbP)$ and $\{\tg_i\}_{i=1}^K$ orthonormal in $L^2(\cX, \bbP_n)$. So $g_1$ is orthogonal to $g_2,\ldots,g_K$. At the same time, from the diagonal structure of $\Sigma$, we know $\tg_1$ is orthogonal to $g_2,\ldots,g_K$ as well. This suggests $g_1$ is collinear with $\tg_1$. On top of that, since the diagonal entries of $\Sigma$ are real positive values, we know $\tg_1=g_1/\|g_1\|_{L^2(\cX, \bbP_n)}$. This in fact holds for all $i \in [K]$, i.e. $\tg_i=g_i/\|g_i\|_{L^2(\cX, \bbP_n)}$. Taking $Q= AB^*$, which is a unitary matrix, we have
	\begin{equation}
	\label{eq:proscute}
	\| Q -W^*\hV_1 \|_2 = \| \Sigma - I \|_2 = \max_{i \in [K]} |1-\ip{\tg_i,g_i}| = \max_{i \in [K]} | 1- \frac{1}{\|g_i\|_{L^2(\cX, \bbP_n)}} |.
	\end{equation}
	
	By our assumption on the singular values, we know for all $i \in [K]$, $\frac{1}{\|g_i\|_{L^2(\cX, \bbP_n)}}=\ip{\tg_i,g_i} \leq 2$. Note that for $x \geq \frac{1}{2}$, $|1-\frac{1}{x}|\leq 2 |x-1| \leq 2 |x^2-1|$, we see
	\begin{equation}
	\max_{i \in [K]} | 1- \frac{1}{\|g_i\|_{L^2(\cX, \bbP_n)}} | \leq 2 \max_{i \in [K]} | 1- \|g_i\|_{L^2(\cX, \bbP_n)}^2 | = 2\max_{i \in [K]} \big| \frac{1}{n}\sum_{j=1}^{n} |g_i(X_j)|^2 - \bbE |g_i(X)|^2 \big|.
	\label{eq:anglechain}
	\end{equation}
	
	Since $g_i$'s rely on the samples $\{X_i\}_{i=1}^n$, they are random. What they have in common is they have bounded $\cH^s$ norm, which is because
	\[
	\|g_i\|_{\cH^s} = \|WAe_i\|_{\cH} \leq \|W\|_{2 \to \cH^s} \|A\|_2 \|e_i\|_2= \|W\|_{2 \to \cH^s}.
	\]
	
	Therefore,
	\begin{equation}
	\max_{i \in [K]} \big| \frac{1}{n}\sum_{j=1}^{n} |g_i(X_j)|^2 - \bbE |g_i(X)|^2 \big| \leq \sup_{\|g\|_{\cH^s}\leq \|W\|_{2 \to \cH^s}}\Big|P_n|g|^2 - P|g|^2 \Big| = \|W\|_{2 \to \cH^s}^2 \sup_{{\|g\|_{\cH^s}\leq 1}} \Big|P_n|g|^2 - P|g|^2 \Big|.
	\label{eq:supchain}
	\end{equation}
	Chaining \eqref{eq:supchain} and \eqref{eq:anglechain} completes the proof. 
\end{proof}

Using Dudley inequality and standard results on the covering number in Sobolev space, we can show (with proof in the appendix)
\begin{lemma}
	For our choice of $s$, we have with probability $1-4\exp(-\tau)$
	\[
	\sup_{{\|g\|_{\cH^s}\leq 1}} \Big|P_n|g|^2 - P|g|^2 \Big|  \leq \frac{C_{19}+C_{20} \sqrt{\tau}}{\sqrt{n}}.
	\]
	\label{lm:ulln}
\end{lemma}

We are now ready to prove Theorem \ref{thm:spec-consis}.
\begin{proof}[proof of theorem \ref{thm:spec-consis}]
	For fixed sample size $n$, let $\cE_{n,1}$ be the event when the concentration in Proposition \ref{prop:T-concen} holds, and let $\cE_{n,2}$ be the event when the concentration in Lemma \ref{lm:ulln} holds. From now on, we condition on the intersection $\cE_{n,1} \cap \cE_{n,2}$, which happens with probability greater than of equal to $1 - 8e^{-\tau}$. 
	
	First of all, on this event, we know Proposition \ref{prop:sampling-error} also holds. We thus have
	\[
	\|V_1 - W\|_{2 \to \infty} \leq C_{18} \frac{\sqrt{\tau}}{\sqrt{n}}.
	\]
	So $\|V_1\|_{2 \to \infty}$ is close to $\|W\|_{2 \to \infty}$. Since in Theorem \ref{thm:spec-consis}, $n/\tau \geq C_4$ and we have the freedom of choosing $C_5$, we can set $C_5$ large enough such that $\|W\|_{2\to \infty} \leq 2 \|V_1\|_{2\to \infty}$. Imitating the proof of Proposition \ref{prop:sampling-error}, we can similarly show $\|V_1-W\|_{2\to \cH^s}$ is on the order of $\sqrt{\tau}/\sqrt{n}$. We can thus assume $C_4$ is also large enough to ensure $\|W\|_{2\to \cH^s} \leq 2 \|V_1\|_{2\to \cH^s}$.
	
	Meanwhile, due to the uniform law of large numbers in Lemma \ref{lm:ulln}, we can always let $\|g_i\|_{L^2(\cX, \bbP_n)}$ be greater than $1/2$ by setting $C_4$ large enough. The condition on singular values in Lemma \ref{lm:bound-unitary} is thus satisfied and from it we see there exists unitary matrix $Q$ such that
	\[
	\|Q-W^*\hV_1\|_2 \leq 2 \|W\|_{2 \to H^s}^2 \frac{C_{19}+C_{20} \sqrt{\tau}}{\sqrt{n}} \leq 2 \|W\|_{2 \to H^s}^2 \frac{(C_{19}+C_{20}) \sqrt{\tau}}{\sqrt{n}},
	\]
	where we assumed $\tau \geq 1$. Since for concentration results like Theorem \ref{thm:spec-consis} to be meaningful, $\tau$ is large anyway, this assumption is harmless. 
	
	Going back to \eqref{eq:QandWV}, we have
	\begin{align}
	\|V_1 - \hV_1 Q \|_{2 \to \infty} &\leq \|V_1 - W \|_{2 \to \infty} + \|W \|_{2 \to \infty} \| Q^T -W^*\hV_1 \|_2 \\
	&\leq C_{18}\frac{\sqrt{\tau}}{\sqrt{n}} + 2\|W\|_{2 \to \infty} \|W\|_{2 \to H^s}^2 \frac{(C_{19}+C_{20}) \sqrt{\tau}}{\sqrt{n}} \\
	&\leq  C_{18}\frac{\sqrt{\tau}}{\sqrt{n}} + 16\|V_1\|_{2 \to \infty} \|V_1\|_{2 \to H^s}^2 \frac{(C_{19}+C_{20}) \sqrt{\tau}}{\sqrt{n}}.
	\end{align}
	Setting $C_{3}= C_{18} + 16\|V_1\|_{2 \to \infty} \|V_1\|_{2 \to H^s}^2(C_{19}+C_{20})$ thus completes the proof .
\end{proof}

\section{Discussion}
We would like to first comment on the relationship between Theorem \ref{thm:main-thm} and the concentration of spectral projection (see Proposition 22 in \citet{rosasco2010Learninga}). Our result in fact easily implies the concentration of spectral projections. To see this, simply note the difference in projection can be written as $V_1V_1^* - \tV_1 \tV_1^*$ and apply triangular inequality. We believe it is possible to go from the concentration of spectral projection in Hilbert space $\cH$ to Theorem \ref{thm:main-thm}, but the road is treacherous. On a high level, we need to project an orthonormal bases of the leading invariant space of the perturbed operator to that of the unperturbed operator, and then performing Gram-Schmidt on the projections. During this process, we need to convert back and forth from $\cH$ to $L^2(\cX, \bbP)$ and we foresee countless petty and pesky technical details. But it is our belief that the concentration of spectral projections in $\cH$ induced operator norm is equivalent to Theorem \ref{thm:main-thm}.

We would also like to comment on the generality of the Newton-Kantorovich Theorem. By that, we mean the operator equation \eqref{eq:quad_eq} need not be restricted to the space of $\cL(\bbC^K, \cH)$. We can have slightly altered versions of \eqref{eq:quad_eq} involving $L^2(\cX, \bbP)$, $C_b(\cX)$, or $C_b^1(\cX)$ that induce an invariant subspace and still apply the Newton-Kantorovich Theorem to solve them. For example, we should be able to replace every $\cH$ in this paper with $C_b^1(\cX)$ and remake the proof to make everything go through. A word of caution is that to obtain operator norm convergence from the sample level operator to the population operator, the function space one works with has to have some kind of ``smoothness''. Either the kind of smoothness from an RKHS or the kind from $C_b^1(\cX)$ is fine, but spaces like $C_b(\cX)$ or $L^2(\cX,\bbP)$ where functions may oscillate wildly while still having a small norm are not okay, because adversarial functions can be chosen to ruin operator norm convergence. This point was also mentioned in \citet{vonluxburg2008Consistency}.

Finally, we would like to comment on our complex-valued functions assumption and the fact that Theorem \ref{thm:main-thm} needs an unitary matrix $Q$. We feel like since everything is real, the unitary matrix is an artifact rather than a necessity and our proof could be altered so that only an orthonormal matrix is needed (although we don't know how at the moment). We have also checked that we can get around with real Hilbert or Banach spaces and real-valued functions for almost all lemmas and theorems except for Theorem \ref{thm:rosasco}. But on the brighter side, working with complex numbers makes our result more general and can give us the freedom of using a complex-valued kernel function, although such freedom is rarely taken advantage of in statistics or machine learning.
Last but not least, we wish to point out that due to length constraints, we only did one application which is normalized spectral clustering, but other applications of our general theory are possible. For example, uniform consistency results can be obtained for kernel PCA and the proof of that is much simpler than the proof of normalized spectral clustering. We include such results in the appendix.  
\label{sec:discussion}

% \begin{acks}[Acknowledgments]
% And this is an acknowledgements section with a heading that was produced by the
% $\backslash$section* command. Thank you all for helping me writing this
% \LaTeX\ sample file.
% \end{acks}

%%%%%%%%%%%%%%%%%%%%%%%%%%%%%%%%%%%%%%%%%%%%%%
%% Supplementary Material, if any, should   %%
%% be provided in {supplement} environment  %%
%% with title and short description.        %%
%%%%%%%%%%%%%%%%%%%%%%%%%%%%%%%%%%%%%%%%%%%%%%
% \begin{supplement}
% \stitle{Title of Supplement A}
% \sdescription{Short description of Supplement A.}
% \end{supplement}

% \begin{supplement}
% \stitle{Title of Supplement B}
% \sdescription{Short description of Supplement B.}
% \end{supplement}

\bibliographystyle{imsart-nameyear}
\bibliography{specref}

\appendix
\section{Proofs}
\subsection{Proof of Lemma \ref{lm:induced_hilbert}}
\begin{proof}[Proof of Lemma \ref{lm:induced_hilbert}]
For item 1, we show double inclusions: $V_2^{-1} \cH \subset V_2^*\cH$ and $V_2^*\cH \subset V_2^{-1} \cH$. First, for any $l \in V_2^{-1} \cH$, we know $V_2 l \in \cH$, so $V_2^*(V_2 l) \in V_2^*\cH$. Since $l = V_2^*V_2 l $, we see $V_2^{-1} \cH \subset V_2^*\cH$. Second, for any $l \in V_2^* \cH$, suppose without loss of generality that $l = V_2^*h$ for some $h \in \cH$. Note $V_2 V_2^* h = h - V_1V_1^*h$. Since $f_1,\ldots,f_K \in \cH$ by assumption, we know $V_1V_1^*h \in \cH$, so $V_2l =h - V_1V_1^*h \in \cH$. This shows the other inclusion.

For item 2, since $\cH$ is a subspace of $L^2(\cX, \bbP)$, $V_2^*\cH$ is a subspace of $V_2^* L^2(\cX, \bbP)$ which is equal to $l^2$. Checking $(\cdot,\cdot)$ is an inner product is routine. We next show the completeness of $V_2^{-1} \cH$. Let $\{b_i\}_{i=1}^\infty$ be a Cauchy sequence in $V_2^{-1}\cH$. Since by definition $\|b\|_{V_2^{-1}\cH}=\|V_2b\|_\cH$, the sequence $\{V_2b_i\}_{i=1}^\infty$ is Cauchy in $\cH$. Suppose $V_2b_i \xrightarrow{\cH} y$ for some $y \in \cH$. Since by assumption $\|\cdot\|_\cH$ norm is stronger than $\|\cdot\|_{\infty}$ which is in turn stronger than $\|\cdot\|_{L^2}$, we know $V_2b_i \xrightarrow{L^2} y$. Since the range of $V_2$ is closed in $L^2(\cX, \bbP)$, we know $y$ is also in its range and $y = V_2V_2^*y$. Note $\|b_i - V_2^* y\|_{V_2^{-1}\cH}=\|V_2 b_i - V_2V_2^*y \|_\cH$ and the right hand side converges to zero because $V_2b_i \xrightarrow{\cH} y$, we see the space of $V_2^{-1}\cH$ is indeed complete.

For item 3, for any $\alpha \in \bbC^K$ with $\|\alpha\|_2=1$, we have
\[
\|V_1\alpha \|_{\cH} = \| \sum_{i=1}^{K} \alpha_i f_i \|_{\cH} \leq \sum_{i=1}^{K} |\alpha_i| \| f_i \|_{\cH} \leq \sqrt{K} \max_{i \in [K]} \| f_i \|_{\cH}.
\]
This shows $\|V_1 \|_{2 \to \cH} \leq \sqrt{K} \max_{i \in [K]} \| f_i \|_{\cH}$. As is noted in the proof for item 2, $\|\cdot\|_{\cH}$ norm is stronger than the $\|\cdot\|_{L^2}$ norm, i.e. $\|h\|_{L^2} \leq C_\cH \|h\|_{\cH}$ for some constant $C_\cH$ and $\forall h \in \cH$. Therefore
\[
\|V_1^*h\|_2^2 = \sum_{i=1}^K \ip{f_i,h}_{L^2}^2 \leq \sum_{i=1}^K \|f_i\|_{L^2}^2 \|h \|_{L^2}^2 \leq C_\cH^2 K \|h\|_{\cH}^2.
\]
We thus see $\|V_1^*\|_{\cH \to 2} \leq C_\cH \sqrt{K}$. The fact that $\|V_2\|_{V_2^{-1}\cH \to \cH} =1$ is a simple consequence of $\|b\|_{V_2^{-1}\cH}=\|V_2b\|_\cH$ for $\forall b \in V_2^{-1}\cH$. Finally, we have for $\forall h \in \cH$
\[
\|V_2^*h\|_{V_2^{-1}\cH} = \|V_2V_2^*h\|_{\cH} = \|(I-V_1V_1^*)h\|_{\cH} \leq (1+\|V_1 \|_{2 \to \cH} \|V_1^*\|_{\cH \to 2})\|h\|_{\cH}.
\]
It thus follows $\|V_2^*\|_{\cH \to V_2^{-1}\cH} \leq 1 + C_\cH K\max_{i \in [K]} \| f_i \|_{\cH}$.
\end{proof}

\subsection{Proof of a lemma used in proving Theorem \ref{lm:nk}}
\label{sec:lmnk}
In the proof of Theorem \ref{lm:nk}, we used to following lemma. We now state and prove it.
\begin{lemma}
	The operator $\cT$ defined as $\cT:\cE \mapsto \cE$ as $\cT(Y):=T_{22}Y - YT_{11}$ is one-to-one and onto. Moreover, $\inf_{\|Y\|_{HS}=1} \big\| T_{22}Y-YT_{11} \big\|_{HS} > 0$.
\end{lemma}
\begin{proof}
	First, we note
	\begin{align*}  
	T_{11} \colon \bbR^K & \longrightarrow \bbR^K\\
	(a_1,a_2,\ldots,a_K) &\longmapsto (\lambda_1a_1,\lambda_2a_2,\ldots,\lambda_Ka_K)\\
	\end{align*}
	To show $\cT$ is one-to-one and onto, it suffices to show for any $g \in \cE$, there exists a unique $y \in \cE$ such that $g=\cT(y)$. Denote the standard orthonormal basis in $\bbR^K$ by $\{e_i\}_{i=1}^K$. Due to the diagonal structure of $T_{11}$, we see $(\cT y)(e_i)=(T_{22}-\lambda_i I)ye_i$. So to show the existence and uniqueness of $y \in \cE$ such that $g=\cT(y)$, it suffices to show for $\forall i \in [K]$, there exists a unique $ye_i$ such that $(T_{22}-\lambda_i I)ye_i=ge_i$. To this end, it suffices to show $\lambda_i$ is in the resolvent of $T_{22}$. This is indeed true because 1) $T_{22}$ is a compact operator from $\tl^2$ to $\tl^2$; 2) $\sigma(T_{22}) \subset \{\lambda_{K+1}, \lambda_{K+2}, \ldots\} \cup \{0\}$. The second point is obvious and the first point follows from
	\[
	\|T_{22}\|_{HS} = \|F_{\perp}^*TF_{\perp}\|_{HS} \leq \|F_{\perp}^*\|_{op}\|T\|_{HS}\|F_{\perp}\|_{op} \leq \|T\|_{HS}.
	\]
	
	Next, we show $\cT$ is a bounded operator. Once $\cT$ is bounded, since $\cT$ is one to one and onto and $\cE$ is a Banach space, we know from bounded inverse theorem that $\cT^{-1} \in \cL(\cE)$, which is equivalent to $\inf_{\|Y\|_{HS}=1} \big\| T_{22}Y-YT_{11} \big\|_{HS} > 0$.
	
	The operator $\cT$ is indeed bounded because
	\[
	\big\| T_{22}Y-YT_{11} \big\|_{HS} \leq \|T_{22}\|_{HS}\|Y\|_{HS} + \|T_{11}\|_{HS}\|Y\|_{HS} \leq 2 \|T\|_{HS} \|Y\|_{HS}.
	\]
\end{proof}

We remark that when $T_{22}$ is self-adjoint, the proof of this lemma will be greatly simplified. In fact, we have
\begin{align*}
	\big\| T_{22}Y-YT_{11} \big\|_{HS} \geq \big\|YT_{11} \big\|_{HS} - \big\| T_{22}Y \big\|_{HS}  \geq \lambda_K \big\|Y \big\|_{HS} - \big\|T_{22} \big\|_{op}\big\|Y \big\|_{HS} 
\end{align*}
Since $T_{22}$ is self-adjoint, its operator norm is its largest eigenvalue, which is $\lambda_{K+1}$. We see immediately in this case that $\inf_{\|Y\|_{HS}=1} \big\| T_{22}Y-YT_{11} \big\|_{HS} \geq \lambda_K - \lambda_{K+1}$, so the eigengap is recovered. In unnormalized spectral clustering, where $\cH$ is set to be the RKHS associated with the kernel function, we claim $T_{22}$ is self-adjoint.

\subsection{Proof of Lemma \ref{lm:ulln}}

We see at the core of Lemma \ref{lm:ulln} is some uniform law of large number over the unit ball in $\cH^s$. We need the following two lemmas in the proof, the first is from \citet{cucker2002} (Proposition 6), and the second is from \citet{vershynin2018} (Theorem 8.1.6).

\begin{lemma}
	Denote $\cG = \{g \,|\, g \in \cH^s, \|g\|_{\cH^s} \leq 1 \}$. When $s > p/2$, for all $\epsilon > 0$,
	\[
	\log \cN(\cG, \| \cdot \|_\infty, \epsilon) \leq \Big(\frac{C}{\epsilon}\Big)^{p/s} + 1
	\]
	for some constant $C$.
	\label{lm:coveringnumber}
\end{lemma}

\begin{lemma}
	Let $(X_t)_{t \in T}$ be a random process on a metric space $(T,d)$ with sub-gaussian increments, i.e.
	\[
	\|X_t - X_s\|_{\Psi_2} \leq  K d(t,s) \text{   for all  } t,s\in T.
	\] 
	Then, for every $u \geq 0$, the event
	\[
	\sup_{t \in T}X_t \leq C K \Big[ \int_0^\infty \sqrt{\log \cN(T,d,\epsilon)} d\epsilon + u\cdot \text{diam}(T)\Big]
	\]
	holds with probability at least $1 - 2exp(-u^2 )$.
\end{lemma}

We now prove Lemma \ref{lm:ulln}. We refer readers who are unfamiliar with the arguments below to the proof of Theorem 8.2.3 in \citet{vershynin2018}.
\begin{proof}[Proof of Lemma \ref{lm:ulln}]
	We first show on $\cG = \{g \,|\, g \in \cH^s, \|g\|_{\cH^s} \leq 1 \}$, the random process $P_n|g|^2-P|g|^2$ has sub-gaussian increments. For fixed $f,g \in \cG$, we have
	\[
	\|P_nf\bar{f}-Pf\bar{f}-P_ng\bar{g}+Pg\bar{g}\|_{\Psi_2}=\frac{1}{n}\|\sum_{i=1}^{n} Z_i \|_{\Psi_2} \text{  where  }Z_i = (f\bar{f}-g\bar{g})(X_i) - \bbE (f\bar{f}-g\bar{g})(X).
	\]
	So $Z_i$'s are independent and mean zero. It thus follows
	\[
	\|P_nf\bar{f}-Pf\bar{f}-P_ng\bar{g}+Pg\bar{g}\|_{\Psi_2} \leq \frac{C_{21}}{n} \big( \sum_{i=1}^{n} \|Z_i\|_{\Psi_2}^2 \big)^{1/2}.
	\]
	By the centering lemma, we know 
	\[
	\|Z_i\|_{\Psi_2} \leq C_{22} \|f(X_i)\average[2.5]{f(X_i)} - g(X_i)\average[2.5]{g(X_i)}\|_{\Psi_2}.
	\]
	Note that because of the embedding, we have
	\[
	\|f\bar{f}-g\bar{g}\|_\infty \leq 
	\|f(\bar{f}-\bar{g})\|_\infty + \|(f-g)\bar{g}\|_\infty \leq
	\|f\|_\infty \|f-g\|_\infty + \|g\|_\infty \|f-g\|_\infty \leq 2C_6 \|f-g\|_{\infty}.
	\]
	The random variable $f(X_i)\average[2.5]{f(X_i)} - g(X_i)\average[2.5]{g(X_i)}$ is thus bounded. Since bounded random variables have bounded $\Psi_2$ norm, we see
	\[
	\|f(X_i)\average[2.5]{f(X_i)} - g(X_i)\average[2.5]{g(X_i)}\|_{\Psi_2} \leq 2 C_6 C_{23}  \|f-g\|_{\infty}.
	\]
	Putting pieces together, we have
	\[
	\|P_nf^2-Pf^2-P_ng^2+Pg^2\|_{\Psi_2} \leq \frac{C_{24}}{\sqrt{n}} \|f-g\|_\infty.
	\]
	
	Next, it is easy to examine that $\text{diam}(\cG) \leq 2C_6$ and $\int_0^{2C_6} \sqrt{\log \cN(T,d,\epsilon)} d\epsilon < \infty$ (because our choice of $s=\floor{p/2}+1$ and Lemma \ref{lm:coveringnumber}). It thus follows that the event
	\[
	\sup_{g \in \cG} P_n g^2 -P g^2 \leq \frac{C_{19}+C_{20} u}{\sqrt{n}}
	\] 
	holds with probability $1-2exp(-u^2)$.
	
	By the exactly same argument, we can also show the event
	\[
	\sup_{g \in \cG} P g^2 -P_n g^2 \leq \frac{C_{19}+C_{20} u}{\sqrt{n}}
	\] 
	holds with probability $1-2exp(-u^2)$. Taking union bound, we see
	\[
	\sup_{g \in \cG} \big| P g^2 -P_n g^2 \big| \leq \frac{C_{19}+C_{20} u}{\sqrt{n}}
	\] 
	holds with probability $1-4exp(-u^2)$. Rewrite $u^2$ as $\tau$ and the proof is complete.
\end{proof}

\section{Application to kernel PCA}
To further demonstrate the usage of the general theory, we apply it to kernel principal component analysis (kernel PCA) in this section.
In kernel PCA, we start from a metric space $\cX$, a probability measure $\bbP$ on $\cX$, and a continuous positive definite kernel function $k:\cX\times\cX\to\reals$. After observing samples $X_1,\dots,X_n\overset{\iid}{\sim} \bbP$, we are interested in matrix $K_n\in\reals^{n\times n}$ of their pairwise similarities: $K_n = \begin{bmatrix}\frac1nk(X_i,X_j)\end{bmatrix}_{i,j = 1}^n$, assuming our data mapped into the feature space is centered.
Since $K_n$ is symmetric and positive semi-definite, it has an eigenvalue decomposition.
We denote the eigenpairs by $(\hlambda_k, v_k)$ and sort the eigenvalues in descending order:
\[
\hlambda_1\ge\dots\ge\hlambda_n \ge 0.
\]
The eigenvectors $v_k$ are normalized to have $\|v_k\|_2=\sqrt{n}$.
Then the matrix $V = [v_1, \cdots, v_K] \in \mathbb{R}^{n \times K}$ consists of the leading $K$ principal components.

Let $\cH$ be the RKHS associated with kernel $k(\cdot,\cdot)$.
Recall that the \emph{tensor product} of $a, b\in \mathcal{H}$,
\[
\begin{aligned}
a \otimes b: \mathcal{H} & \rightarrow \mathcal{H} \\
f & \mapsto \langle b, f\rangle_{\mathcal{H}} a,
\end{aligned}
\]
is a linear operator.
Then the operator counterpart of $\tK_n$ is the {\it empirical covariance operator}  
\begin{align}
	\Sigma_{n} = \frac{1}{n} \sum_{i=1}^{n} k_{X_i} \otimes k_{X_i},
\end{align}
where $k_{X_{i}} = k(\cdot, X_i)$ is the corresponding feature in $\mathcal{H}$ of $X_i$ ($i = 1, \cdots, n$) under the feature map $x \mapsto k(\cdot, x)$. 

It turns out that the eigenvalues and eigenvectors of $K_{n}$ and $\Sigma_n$ are closely related. To formulate this relationship, let us define the restriction operator $\zeta : \mathcal{H} \to \mathbb{R}^{n}$ by $\zeta f=\frac{1}{\sqrt{n}}(f(X_{1}), \cdots, f(X_{n}))^T$. Then verifiably, the adjoint of $\zeta$, $\zeta^{*} : \mathbb{R}^{n} \mapsto \mathcal{H}$ is given by $\zeta^{*} \alpha=\frac{1}{\sqrt{n}} \sum_{i=1}^{n} \alpha_{i} k_{X_{i}}$, where $\alpha = (\alpha_1, \cdots, \alpha_n)^T$. The eigenvalues and eigenvectors(functions) of $\Sigma_n$ and $K_n$ are closely
related in the following sense.

\begin{lemma}
\label{lem:5.1}
	The following facts hold true:
	\begin{enumerate}
		\item $K_n = \zeta \zeta^*$ and $\Sigma_n = \zeta^* \zeta$;
		\item If $(\hlambda,f)$ is a non-trivial eigenpair of $\Sigma_n$ (i.e. $\hlambda \neq 0$), then $(\hlambda, \zeta f)$ is an eigenpair for $K_n$.
		\item If $(\hlambda, v)$ is a non-trivial eigenpair of $K_n$, then $(\hlambda, \hf)$, where
		\begin{equation}
		\label{eq:eigenfunction-KPCA}
		    \hf(x) = \frac{1}{\hlambda n}\zeta^\star v = \frac{1}{\hlambda n} \sum_{i=1}^n k(x, X_i) v_i
		\end{equation}
		is an eigenpair for $\Sigma_n$ with $\hf \in \cH$. Moreover, this choice of $\hf$ is such that $\|\hf\|_{L^2(\cX, \bbP_n)} = 1$ and the restriction of $\hf$ onto sample points agrees with $v$, i.e. $\zeta \hf = v$. 
	\end{enumerate}
\end{lemma}

\begin{proof}
For item one, it is easy to verify $\zeta \zeta^*:\bbR^n \mapsto \bbR^n$ is the linear transformation defined by $K_n$. For the other half, noting that
\[
\zeta^* \zeta f = \frac{1}{n} \sum_{i=1}^{n} f(X_i)k(x,X_i).
\]
At the same time, by the reproducing property
\[
\Sigma_n f = \frac{1}{n} \sum_{i=1}^{n} \ip{k_{X_i},f}_\cH k_{X_i} = \frac{1}{n} \sum_{i=1}^{n} f(X_i)k(x,X_i).
\]
We thus conclude the two are equal.

For item two, since by assumption $\zeta^* \zeta f = \hlambda f$, we have $K_n \zeta f = \zeta \zeta^* \zeta f = \hlambda \zeta f$, which is exactly what the statement suggests.

For item three, if $(\hlambda,v)$ is an eigenpair of $K_n$, we check that $(\hlambda,\hf)$ is an eigenpair of $\Sigma_n$:
\[
\begin{aligned}
\Sigma_n\hf &= \frac1n\sum_{i=1}^n (k_{X_i}\otimes k_{X_i})\left(\frac{1}{\hlambda n}\sum_{j=1}^n k_{X_j} v_j\right) \\
&= \frac{1}{\hlambda n} \sum_{i=1}^n \left(k_{X_i}\sum_{j=1}^n k(X_i, X_j) v_j\right)
= \frac{1}{\hlambda n} \sum_{i=1}^n k_{X_i} [K_n v]_i
= \frac{1}{\hlambda n} \sum_{i=1}^n k_{X_i} \hlambda v_i =\hlambda\hf.
\end{aligned}
\]
Moreover, $\hf$ is a linear combination of $k_{X_i}$ and therefore belongs to $\cH$.
\end{proof}

The population version of $\Sigma_n$ is the {\it covariance operator} 
\[
\Sigma= \bbE k_X \otimes k_X,
\]
where $k_{X} = k(\cdot, X)$, $d = \mathbb{E} k_{X}$, and $X \sim \mathbb{P}$.
We will later justify the expectation of such random elements in an appropriate Hilbert space.
Under appropriate assumptions, it can be shown that we can choose $\{f_i\}_{i=1}^K$, the top $K$ eigenfunctions of $\Sigma$, to be real-valued and orthonormal in $L^2(\mathcal{X}, \mathbb{P})$.
Then we can define $V_{1} : \mathbb{C}^{K} \rightarrow \mathcal{H}$ as $V_{1} \alpha=\sum_{i=1}^{K} \alpha_{i} f_{i}$. Similarly, we can define $\widehat{V}_{1}$ with $\{\widehat{f}_{i}\}_{i=1}^{K}$, the extension of top $K$ orthonormal eigenvectors of $\widetilde{K}_{n}$ according to \eqref{eq:eigenfunction-KPCA}. We are now ready to apply our general theory to prove the following result, which is similar to Theorem \ref{thm:main-thm}.

\begin{theorem}
\label{thm:main-thm-2}
Under the general assumptions defined below, there exists $C_6, C_7$ that are determined by $\mathcal{X}, \mathbb{P}, k(\cdot, \cdot)$ such that whenever sample size $n \geq C_6 \tau$ for some $\tau > 1$, we have with confidence $1 - 6e^{-\tau}$,
\begin{equation}
    \inf \left\{\|V_{1}-\widehat{V}_{1} Q\|_{2 \rightarrow \infty} : Q \in \mathbb{U}^{K}\right\} \leq C_{7} \frac{\sqrt{\tau}}{\sqrt{n}}.
\end{equation}
\end{theorem}

The general assumptions referred to in Theorem 5.2 are

\textsc{General Assumptions.} The set  $\mathcal{X}$ is a separable topological space. The kernel $k: \mathcal{X} \times \mathcal{X} \rightarrow \mathbb{R}$ is continuous, symmetric, positive semi-definite, and
\begin{equation}\label{eq:assumption-KPCA}
    \sup_{x \in \mathcal{X}} k(x, x) < \infty.
\end{equation}
Treated as an operator from $\cH$ to $\cH$, the eigenvalues of $\Sigma$ satisfy $\lambda_{1} \geq \ldots \geq \lambda_{K}>\lambda_{K+1} \geq \ldots \geq 0$. The top $K$ eigenfunctions of $\Sigma_n$, $\{f_i\}_{i=1}^K\subset C_b(\cX)$.

Condition \eqref{eq:assumption-KPCA} ensures that all of the operators we are working with are Hilbert-Schmidt, and further guarantees concentration of bounded random elements in the Hilbert space of Hilbert-Schmidt operators. Separability of $\mathcal{X}$ and continuity of $k(\cdot, \cdot)$ assures that the RKHS $\mathcal{H}$ is separable by Lemma 4.33 of \cite{steinwart2008support}.

\subsection{Overview of the proof}
The proof of Theorem \ref{thm:main-thm-2} is simpler than that of Theorem \ref{thm:main-thm} because we can work with the reproducing kernel Hilbert space $\cH$ associated with $k(\cdot,\cdot)$ directly. We shall show that $\Sigma_n - \Sigma$, as an operator from $\cH$ to $\cH$, has Hilbert-Schmidt norm tending to zero as $n$ goes to infinity. Recall that the columns of $\hV_1$ are only orthonormal in $L^{2}(\mathcal{X}, \mathbb{P}_{n})$ but the general theory requires $\tV_1$ which has columns orthonormal in $L^{2}(\mathcal{X}, \mathbb{P})$. Similar to the challenge of proving Theorem \ref{thm:main-thm}, we also need to deal with the error induced by the distinction between orthonormality in $L^{2}(\mathcal{X}, \mathbb{P}_{n})$ and $L^{2}(\mathcal{X}, \mathbb{P})$, which is one of the key steps for the proof of Theorem \ref{thm:main-thm}.

The rigorous treatment shall be presented in five parts. In part one, we introduce $\mathcal{L}_{HS}(\mathcal{H})$, the Hilbert space of Hilbert-Schmidt operators from $\mathcal{H}$ to $\mathcal{H}$ that we work with, and justify the random elements in $\mathcal{H}$ and $\mathcal{L}_{HS}(\mathcal{H})$. In part two, we build up concentration results in the Hilbert space $\mathcal{L}_{HS}(\mathcal{H})$. In part three, we check the remaining conditions required by our general theory. In part four, we put all the above ingredients together and apply our general theory to complete the proof.

\subsection{Part one: the space $\mathcal{L}_{HS}(\mathcal{H})$ and random elements in $\mathcal{H}$ and $\mathcal{L}_{HS}(\mathcal{H})$}
The space $\mathcal{L}_{HS}(\mathcal{H})$ collects all of the Hilbert-Schmidt operators from $\cH$ to $\cH$, which is also a Hilbert space. Before justifying the covariance operator, we first define the mean element in $\cH$ and the cross-covariance operator in $\mathcal{L}_{HS}(\widetilde{\cH}, \mathcal{H})$.

The mean element in $\cH$ is defined as $\mu_X = \bbE_{X\sim\bbP}[k(\cdot, X)] \in \cH$ such that $\left\langle f, \mu_k\right\rangle_{\cH}=\mathbb{E}_{X\sim\bbP}[f(X)]$ for any $f\in\cH$. Let $(\cY, \cB_{\cY}, \bbQ)$ be another probability space and $\widetilde{\cH}$ is a reproducing kernel Hilbert space associated with kernel $\widetilde{k}(\cdot,\cdot)$ containing $\widetilde{k}(\cdot,y)$, $y\in\cY$ as its elements. The cross-covariance operator $C_{X,Y} = \bbE_{X\sim\bbP,Y\sim\bbQ}[k(\cdot,X)\otimes \widetilde{k}(\cdot,Y)]$ is a Hilbert-Schmidt operator from $\widetilde{\cH}$ to $\mathcal{H}$ such that for any $f\in\cH$ and $g\in\widetilde{\cH}$, $\langle f, C_{X,Y} g\rangle_\cH = \bbE_{X\sim\bbP,Y\sim\bbQ}[f(X)g(Y)]$. Finally, the covariance operator is $\Sigma = \bbE k_X\otimes k_X = C_{X,X} \in \cL(\cH)$.

The covariance operator is indeed Hilbert-Schmidt because $\|\bbE k_X\otimes k_X\|_{HS}\leq \bbE[\|k_X\otimes k_X\|_{HS}^2] = \bbE[k(X,X)^2] \leq \left(\sup_{x\in\cX}k(x,x)\right)^2 < \infty$ by the general assumption.

\subsection{Part two: concentration in the Hilbert space $\mathcal{L}_{HS}(\mathcal{H})$}
In part one we have shown that $\Sigma$ is a Hilbert-Schmidt operators from $\mathcal{H}$ to $\mathcal{H}$. In this subsection, we show concentration of $\|\Sigma_n - \Sigma\|_{HS}$ by using Lemma \ref{lm:Hilbert-concentration}.

\begin{lemma}
\label{lem:concentration-in-HS-operator-space}
Under the general assumptions, with probability $1-2 e^{-\tau}$, we have
\begin{equation}
    \left\| \frac{1}{n} \sum_{i=1}^{n} k_{X_{i}} \otimes k_{X_{i}} - \mathbb{E} k_{X} \otimes k_{X} \right\|_{HS} \leq C_{23} \frac{\sqrt{\tau}}{\sqrt{n}}
\end{equation}
for some constant $C_{23}$.
\end{lemma}

\begin{proof}
We denote $M:= \sqrt{\sup _{x \in \mathcal{X}} k(x, x)}<\infty$. For any $i \in[n]$, we have
\[
\begin{aligned}
    \| k_{X_{i}} \otimes k_{X_{i}}-\mathbb{E} k_{X} \otimes k_{X} \|_{HS} &\leq \| k_{X_{i}} \otimes k_{X_{i}} \|_{HS} + \|\mathbb{E} k_{X} \otimes k_{X} \|_{HS} \\
    &= \|k_{X_i}\|_\mathcal{H}^2 + \|\mathbb{E} k_{X} \otimes k_{X} \|_{HS} \\
    &= k(X_i, X_i) + \|\mathbb{E} k_{X} \otimes k_{X} \|_{HS} \\
    &\leq M^2 + \|\mathbb{E} k_{X} \otimes k_{X} \|_{HS},
\end{aligned}
\]
which implies $k_{X_{i}} \otimes k_{X_{i}}-\mathbb{E} k_{X} \otimes k_{X}$'s are bounded zero mean independent random variables in $\mathcal{L}_{HS}(\mathcal{H}, \mathcal{H})$.

We can take $C_{23} = \sqrt{2}(M^2 + \|\mathbb{E} k_{X} \otimes k_{X} \|_{HS})$. Then by Lemma \ref{lm:Hilbert-concentration}, we have
\begin{equation*}
     \left\| \frac{1}{n} \sum_{i=1}^{n} k_{X_{i}} \otimes k_{X_{i}} - \mathbb{E} k_{X} \otimes k_{X} \right\|_{HS} =  \left\| \frac{1}{n} \sum_{i=1}^{n} (k_{X_{i}} \otimes k_{X_{i}} - \mathbb{E} k_{X} \otimes k_{X}) \right\|_{HS} \leq C_{23} \frac{\sqrt{\tau}}{\sqrt{n}}
\end{equation*}
with probability $1-2 e^{-\tau}$.
\end{proof}

\subsection{Part three: checking conditions for general theory}
\begin{lemma}
	Under the general conditions, the following facts hold true:
	\begin{enumerate}
		\item The reproducing kernel Hilbert space $H$ is a subspace of $L^2(\cX, \bbP)$.
		\item The $\cH$ norm $\|\cdot \|_{\cH}$ is stronger than infinity norm.
		\item Both $\Sigma, \Sigma_n$ are Hilbert-Schmidt from $\cH$ to $\cH$.
		\item All eigenvalues of $\Sigma$ (counting multiplicity) can be arranged in a decreasing (possibly infinite) sequence of non-negative real numbers $\lambda_1 \geq \lambda_2 \geq \ldots \geq \lambda_K > \lambda_{K+1} \geq \ldots \geq 0$ with a positive gap between $\lambda_K$ and $\lambda_{K+1}$.
		\item The top $K$ eigenfunctions $\{f_i\}_{i=1}^K \subset \cH$ and can be picked to form an orthonormal set of functions in $L^2(\cX, \bbP)$.
		\item $\Sigma_n$ has a sequence of non-increasing, real, non-negative eigenvalues.
	\end{enumerate}
\label{lm:check-for-gen-theory-2}
\end{lemma}
\begin{proof}
Recall that the general conditions entail that $M:= \sqrt{\sup _{x \in \mathcal{X}} k(x, x)}<\infty$. The kernel $k(\cdot,\cdot)$ is therefore a Mercer’s kernel, which satisfies the condition $\int_{\cX\times\cX} k^2(x,y) d\bbP(x) d\bbP(y) < \infty$. Then item 1 is an implication of Mercer’s theorem.

For item 2, for any $f\in\cH$, by the reproducing property and Cauchy-Schwarz inequality, we have $|f(x)|^2 = \langle f, k(\cdot, x) \rangle_{\cH}^2 \leq |k(x,x)|^2  \|f\|_{\cH}^2 \leq M^2 \|f\|_{\cH}^2$. Therefore $\|f\|_{\infty} \leq M\|f\|_{\cH}$.

Item 3 has been checked in the intermediate step of proof of Lemma \ref{lem:concentration-in-HS-operator-space}.

Both item 4 and item 5 are ensured by Mercer's theorem.

Item 6 is true because of the relationship between the spectrum of $\Sigma_n$ and that of the symmetric positive semi-definite kernel matrix $K_n$, which has been checked in Lemma \ref{lem:5.1}.
\end{proof}

\subsection{Part four: putting all ingredients together} To complete the proof of Theorem \ref{thm:main-thm-2}, we have to deal with the error induced by operator $\hV_1$ only having orthonormal columns in $L^{2}(\mathcal{X}, \mathbb{P}_{n})$ but not in $L^{2}(\mathcal{X}, \mathbb{P})$.
This can be accomplished using the same trick as part five of the proof of Theorem \ref{thm:main-thm}. So we omit the technical redundance here.

\end{document}